%% file: puiseuxd3.tex
\documentclass{scrartcl}

\usepackage{amsmath,amsfonts,amsthm,amssymb}%
\usepackage[linesnumbered,vlined,boxed,fillcomment]{algorithm2e}
\SetKw{In}{ in }
\SetKw{And}{ and }
\SetKw{All}{ all }
\SetKw{Or}{ or }
\SetKw{Break}{break}
\SetKw{Continue}{continue}
\SetKw{KwFrom}{from }
\SetKwInput{KwIn}{In}
\SetKwInput{KwOut}{Out}
\makeatletter
\let\oldnl\nl
\newcommand{\nonl}{\renewcommand{\nl}{\let\nl\oldnl}}
\makeatother

\usepackage{enumerate}%
\usepackage{tikz}%
\usepackage{subcaption}

\usepackage[numbers,sort&compress]{natbib}

\usepackage{hyperref}%
\hypersetup{
  colorlinks=true,
  breaklinks=true,
  urlcolor= blue,
  linkcolor= blue,
  bookmarksopen=true,
  pdftitle={fast Puiseux},
  pdfauthor={Adrien Poteaux},
}

\input{commands.tex}

\begin{document}

\title{Computing Puiseux series: a fast divide and conquer algorithm}

\author{%
  Adrien Poteaux,\\%
  {CRIStAL, Universit\'e de Lille}\\%
  {UMR CNRS 9189, B\^atiment M3}\\%
  {59655 Villeneuve d'Ascq, France}\\%
  \texttt{adrien.poteaux@univ-lille.fr}
  \and Martin Weimann\\%
  {GAATI\footnote{Current delegation. Permanent position at LMNO,
     University of Caen-Normandie, BP 5186, 14032 Caen Cedex,
     France.}
, Universit\'e de Polyn\'esie Fran\c{c}aise}\\%
  {UMR CNRS 6139, BP 6570}\\%
  {98702 Faa'a, Polyn\'esie Fran\c{c}aise}\\%
  \texttt{martin.weimann@upf.pf}%
}

\maketitle
\begin{abstract}
  \noindent
  Let $F\in \Ki[X,Y]$ be a polynomial of total degree $\dt$ defined
  over a perfect field $\Ki$ of characteristic zero or greater than
  $\dt$. Assuming $F$ separable with respect to $Y$, we provide an
  algorithm that computes all Puiseux series of $F$ above
  $X=0$ in less than $\Ot(\dt\,\vRF)$ operations in $\Ki$, where
  $\vRF$ is the valuation of the resultant of $F$ and its partial
  derivative with respect to $Y$. To this aim, we use a divide and
  conquer strategy and replace univariate factorisation by dynamic
  evaluation. As a first main corollary, we compute the irreducible
  factors of $F$ in $\Ki[[X]][Y]$ up to an arbitrary precision $X^N$
  with $\Ot(\dt(\vRF+N))$ arithmetic operations. As a second main
  corollary, we compute the genus of the plane curve defined by $F$
  with $\Ot(\dt^3)$ arithmetic operations and, if $\Ki=\Qi$, with
  $\Ot((h+1)\dt^3)$ bit operations using probabilistic algorithms,
  where $h$ is the logarithmic height of $F$.
\end{abstract}


\input{intro/intro}

\input{defpui/defpui}

\input{arnp/arnp}

\input{rnp3/rnp3}

\input{D5/D5}

\input{global/desing}

\input{facto/facto}

\input{conc/conc}

\bibliographystyle{abbrv}{\bibliography{tout}}
\addcontentsline{toc}{section}{References.}

\end{document}

%% file: commands.tex
\newcommand{\hbb}[1]{\ensuremath{\mathbb{#1}}}
\newcommand{\Ai}{\hbb{A}}

\newcommand{\Fi}{\hbb{F}}
\newcommand{\Ki}{\hbb{K}}
\newcommand{\Ni}{\hbb{N}}
\newcommand{\Qi}{\hbb{Q}}
\newcommand{\Ri}{\hbb{R}}
\newcommand{\Zi}{\hbb{Z}}

\newcommand{\hcal}[1]{\ensuremath{\mathcal{#1}}}

\newcommand{\Nc}{\hcal{N}}
\newcommand{\Oc}{\hcal{O}}
\newcommand{\Rc}{\hcal{R}}

\newcommand{\Ns}{\Nc^{\star}}
\newcommand{\Nn}[1][n]{\Nc_{#1}}

\theoremstyle{plain}
\newtheorem{thm}{Theorem}
\newtheorem{prop}{Proposition}
\newtheorem{cor}{Corollary}
\newtheorem{lem}{Lemma}
\theoremstyle{definition}
\newtheorem{dfn}{Definition}
\newtheorem{idea}{Idea}
\newtheorem{xmp}{Example}
\theoremstyle{remark}
\newtheorem{rem}{Remark}

\newcommand{\arnp}{\textnormal{\texttt{Half-RNP}}}
\newcommand{\hrnp}{\textnormal{\texttt{Half-RNP3}}}
\newcommand{\wrnp}{\textnormal{\texttt{MonicRNP}}}
\newcommand{\rnp}{\textnormal{\texttt{RNP}}}
\newcommand{\wdrnp}{\textnormal{\texttt{MonicRNP3}}}
\newcommand{\drnp}{\textnormal{\texttt{RNP3}}}
\newcommand{\wpt}{\textnormal{\texttt{WPT}}}
\newcommand{\quorem}{\textnormal{\texttt{QuoRem}}}
\newcommand{\quo}{\textnormal{\texttt{Quo}}}
\newcommand{\hensel}{\textnormal{\texttt{Hensel}}}
\newcommand{\onestep}{\textnormal{\texttt{HenselStep}}}
\newcommand{\monic}{\textnormal{\texttt{Monic}}}
\newcommand{\bezout}{\textnormal{\texttt{B\'ezout}}}
\newcommand{\factor}{\textnormal{\texttt{Factor}}}
\newcommand{\sqrfree}{\textnormal{\texttt{SQR-Free}}}
\newcommand{\PrimElt}{\textnormal{\texttt{Primitive}}}
\newcommand{\PrimEltB}{\textnormal{\texttt{BivTrigSet}}}
\newcommand{\RNP}{\textnormal{\texttt{RNPuiseux}}}
\newcommand{\NewPol}{\textnormal{\texttt{Polygon-Data}}}
\newcommand{\norm}{\textnormal{\texttt{NormRPE}}}
\newcommand{\Split}{\textnormal{\texttt{Split}}}
\newcommand{\ReducePol}{\textnormal{\texttt{ReducePol}}}
\newcommand{\rmCP}{\textnormal{\texttt{removeCriticalPairs}}}
\newcommand{\desing}{\textnormal{\texttt{Desingularise}}}

\newcommand{\assign}{{\;\leftarrow{}\;}}%

\renewcommand{\O}{\textrm{\Oc}}
\newcommand{\Ot}{\O\tilde\,\,}
\newcommand{\M}{\textup{\textsf{M}}}
\newcommand{\An}{\textup{\textsf{A}}_s}
\newcommand{\dt}{D}
\newcommand{\dx}{{d_X}}
\newcommand{\dy}{{d_Y}}
\newcommand{\height}[1]{\mbox{\rmfamily\upshape ht}(#1)}
\newcommand{\Mg}{M_{\gamma}}
\newcommand{\dg}{d_{P}}

\newcommand{\res}[1][Y]{\text{Res}_{#1}}
\newcommand\disc[1][Y]{{\text{Disc}_{#1}}} 
\newcommand{\val}[1][X]{\ensuremath{\upsilon_{#1}}}
\newcommand{\vF}[1][]{{\ensuremath{\delta_{#1}}}}
\newcommand{\vRF}[1][]{\vF[#1]}
\newcommand{\vinf}[1][F]{\ensuremath{v}}

\newcommand{\eps}{\epsilon}
\newcommand{\lc}[2]{\mbox{\rmfamily lc}_{#1}(#2)}
\newcommand{\algclos}[1]{\overline{#1}}
\newcommand{\tronc}[2]{{\lceil #1 \rceil}^{#2}}
\newcommand{\kgh}{\kappa}
\newcommand{\Gh}{\widehat{G}}
\newcommand{\Gt}{\tilde{G}}
\newcommand{\Ht}{\tilde{H}}
\newcommand{\Ut}{\tilde{U}}
\newcommand{\Vt}{\tilde{V}}
\newcommand{\Ft}{\tilde{F}}
\newcommand{\edge}[3]{#1\,a+#2\,b=#3}
\newcommand{\RPE}{RPE}
\newcommand{\Z}{\underline Z}



%% file: intro/intro.tex
\section{Introduction.}
\label{sec:intro}
This paper provides complexity results for computing Puiseux series of
a bivariate polynomial with coefficients over a perfect field of
characteristic zero or big enough.
\paragraph{Context and main results.} In this paper, $\Ki$ denotes a
perfect field (e.g. $\Ki$ iss a finite or number field), $p$ its
characteristic, $X$ and $Y$ two indeterminates over $\Ki$ and
$F\in\Ki[X,Y]$ a bivariate polynomial primitive and separable in $Y$.
We denote $\dt$ the total degree of $F$, $\dx=\deg_X(F)$ and
$\dy=\deg_Y(F)$; we always assume $p=0$ or $p>\dy$. Let $\algclos\Ki$
be the algebraic closure of $\Ki$ and $\vRF=\val(R_F)$ the
$X$-valuation of the resultant $R_F=\res(F,F_Y)$ of $F$ and its
$Y$-derivative $F_Y$. With our assumption on $p$, the Puiseux theorem
states that for any $x_0\in \algclos{\Ki}$, the roots of $F$ (viewed
as a univariate polynomial in $Y$) may be expressed as fractional
Laurent power series in $(X-x_0)$ with coefficients in
$\algclos{\Ki}$. These are the (classical) \emph{Puiseux
  series}\footnote{terms written \textit{in italics} in this
  introduction are defined in Section \ref{sec:defs} or
  \ref{ssec:trigsets}.} of $F$ above $x_0$, fundamental objects of the
theory of algebraic curves \cite{Wa50,BrKn86}. Many applications are
given in \cite{PoRy15,PoRy12}.

For the computation of \emph{singular parts} of Puiseux series (that
contain the relevant information about the singularities of the
associated curve; remaining terms can be computed up to an arbitrary
precision in quasi-linear time via Newton iterations), we get:

\begin{thm}\label{thm:puidV}%
  There exists an algorithm\footnote{our algorithms are Las Vegas, due
    to the computation of primitive elements; they should become
    deterministic via the preprint \cite{HoLe18}. See Remark
    \ref{rem:deterministic} and Sections \ref{ssec:previous} and
    \ref{ssec:compD5}} that computes singular parts of Puiseux series
  of $F$ above $x_0=0$ in an expected $\Ot(\dy\,\vRF)$ arithmetic
  operations over $\Ki$.
\end{thm}
Here we use the classical $\Ot$ notation that omits logarithmic factors
(see Section \ref{ssec:comp}). This improves the bound
$\Ot(\dy^2\,\vRF)$ of \cite{PoRy15}. From that we deduce:
\begin{thm}\label{thm:puid3}
  There exists an algorithm that computes the singular part of Puiseux
  series of $F$ above \emph{all} critical points in an expected
  $\Ot(\dy^2\dx)\subset\Ot(\dt^3)$ arithmetic operations.
\end{thm}
This improves the bound $\Ot(\dy^2\dx^3)\subset\Ot(\dt^5)$ of
\cite{PoRy08,PoRy11}; note that \cite[Proposition 12]{PoRy15} suggests
a bound $\Ot(\dy^3\dx)\subset\Ot(\dt^4)$.  Via the Riemann-Hurwitz
formula, we get:
\begin{cor}\label{cor:g}
  Assuming $p=0$ or $p>\dt$, there exists an algorithm that computes
  the genus of a given geometrically irreducible algebraic plane curve
  over $\Ki$ of degree $\dt$ in an expected $\Ot(\dt^3)$ arithmetic
  operations.
\end{cor}
Moreover, using the reduction criterion of \cite{PoRy08,PoRy12}, we
can bound the bit complexity of the genus computation (here
$\height{P}$ stands for the maximum between the logarithm of the
denominator of $P$, and the logarithm of the infinite norm of its
numerator):
\begin{cor}\label{cor:gMC}
  Let $\Ki=\Qi(\gamma)$ be a number field, $0<\eps<1$ a real
  number and $F\in\Ki[X,Y]$. Denote $\Mg$ the minimal polynomial of
  $\gamma$ and $w$ its degree. Then there exists a Monte Carlo
  algorithm that computes the genus of the curve $F(X,Y)=0$ with
  probability of error less than $\eps$ and an expected number of word
  operations in:
  \[
  \Ot(\dy^2\dx{}w^2\log^2\epsilon^{-1}[\height{\Mg} +
  \height{F} +1]).
  \]
\end{cor}
With the same notations as in Corollary \ref{cor:gMC}, we have:
\begin{cor}\label{cor:gLV}
  Assuming that the degree of the square-free part of the resultant
  $\res(F,F_Y)$ is known, there exists a Las Vegas algorithm that
  computes the genus of the curve $F(X,Y)=0$ with an expected number
  of word operations in:
  \[
    \Ot(\dy^2\dx{}w^2 [\height{\Mg} + \height{F}+1]).
  \]
\end{cor}
Finally, our algorithm induces a fast analytic factorisation of $F$:
\begin{thm}\label{thm:anfact}
  There exists an algorithm that computes the irreducible analytic
  factors of $F$ in $\Ki[[X]][Y]$ with precision $N\in\Ni$ in an
  expected $\Ot(\dy(\vRF+N))$ arithmetic operations in $\Ki$, plus the
  cost of one univariate factorisation of degree at most $\dy$.
\end{thm}
This has a particular interest with regards to
factorisation in $\Ki[X,Y]$ or $\algclos{\Ki}[X,Y]$: when working
along a critical fiber, one can take advantage of some combinatorial
constraints imposed by ramification when recombining analytic factors
into rational factors \cite{We16}.
\paragraph{Main ideas and organisation of the paper.} Classical
definitions related to Puiseux series and description of the
\emph{rational Newton--Puiseux algorithm} of \cite{Du89} are provided
in Section \ref{sec:defs}. Then, the paper is organised accordingly to
the following main ideas:
\begin{idea}
  \label{id:Weierstrass}
  \textbf{Concentrate on the monic case}. The roots above $(0,\infty)$
  require special care (see Section \ref{ssec:proofD3-general}), This
  is why we use $\vRF=\val(\res(F,F_Y))$ and not $\val(\disc(F))\le \vRF$.
\end{idea}
\begin{idea}
  \label{id:fact}
  \textbf{Use tight truncation bounds} for the powers of $X$ in the
  course of the algorithm. The bound $n=\vRF$ can be reached for
  \emph{some} Puiseux series, but we prove in Section \ref{sec:algo}
  that we can compute at least half of them using a bound
  $n\in{}\O(\vRF/\dy)$.
\end{idea}
\begin{idea}
  \label{id:hens}
  \textbf{A divide and conquer algorithm}. From Idea \ref{id:fact}, we
  prove that $F$ is irreducible (and get its Puiseux series) or get a
  factorisation $F=G\,H \mod X^n$ where $n\in{} \O(\vRF/\dy)$, $G$
  corresponds to the computed Puiseux series, and $H$ satisfies
  $\deg_Y(H)\leq\dy/2$. The fiber $X=0$ being critical, $G(0,Y)$ and
  $H(0,Y)$ are not coprime, and the classical Hensel lemma does not
  apply. But it can be adapted to our case to lift the factorisation
  $F=G\,H$ up to precision $\vRF$. This requires a B\'ezout relation
  $U\,G+V\,H=X^\kgh$ with $\kgh\in\O(\vRF/\dy)$, computed via
  \cite{MoSc16}. Finally, we recursively compute the Puiseux series of
  $H$, defining a divide and conquer algorithm to compute an analytic
  factorisation of $F \mod X^{\vRF+1}$, together with the singular
  parts of its Puiseux series above $x_0=0$. See Section
  \ref{sec:rnp3}.
\end{idea}
\begin{idea}
  \label{id:D5}
  \textbf{We rely on dynamic evaluation}. The next step is to get rid
  of univariate factorisations, which are too expansive for our
  purpose. In Section \ref{sec:D5}, we use dynamic evaluation
  \cite{DeDiDu85,DaMaScXi05} to avoid this bottleneck, leading to work
  over product of fields: we have to pay attention to zero divisors
  and perform suitable splittings when required.
\end{idea}

These ideas allow us to compute the desingularisation of the curve
above all its critical points in Section \ref{sec:desing}. We get a
complexity bound, as good as, up to logarithmic factors, the best
known algorithm to compute bivariate resultants. This is Theorem
\ref{thm:puid3}. 

Finally, we develop a fast factorisation algorithm and prove Theorem
\ref{thm:anfact} in Section \ref{sec:facto}.

To conclude, we add further remarks in Section \ref{sec:conc}, showing
in particular that any Newton--Puiseux like algorithm would not lead
to a better worst case complexity.

\paragraph{A brief state of the art.} In \cite{Du89}, D. Duval defines
the rational Newton--Puiseux algorithm over a field $\Ki$ with
characteristic $0$. From the complexity analysis therein, it takes
less than $\O(\dy^6\,\dx^2)$ operations in $\Ki$ when $F$ is monic (no
fast algorithm is used). This algorithm uses the D5-principle, and can
trivially be generalised when $p>\dy$.

In \cite{PoRy08,PoRy11}, an algorithm with complexity
$\Ot(\dy\,\vRF^2+\dy\,\vRF\,\log(p^c))$ is provided over
$\Ki=\Fi_{p^c}$, with $p>\dy$. From this bound is deduced an algorithm
that computes the singular parts of Puiseux series of $F$ above
\emph{all} critical points in $\Ot(\dy^3\,\dx^2\,\log(p^c))$. In
\cite{PoRy15}, still considering $\Ki=\Fi_{p^c}$, an algorithm is
given to compute the singular part of Puiseux series over $x_0=0$ in
an expected $\Ot(\rho\,\dy\,\vRF+\rho\,\dy\log(p^c))$ arithmetic
operations, where $\rho$ is the number of rational Puiseux expansions
above $x_0=0$ (bounded by $\dy$). These two algorithms use univariate
factorisation over finite fields, thus cannot be directly extended to
the $0$ characteristic case. This also explains why the second result
does not provide an improved bound for the computation of Puiseux
series above \emph{all} critical points.

There are other methods to compute Puiseux series or analytic
factorisation, as generalised Hensel constructions \cite{KaSa99,AlAtMa17}, or
the Montes algorithm \cite{Mo99,BaNaSt13} (which works over general
local fields). Several of these methods and a few others have been
commented in previous papers by the first author
\cite{PoRy12,PoRy15}.  Also, there exist algorithms for the genus
based on linear differential operators and avoiding the computation of
Puiseux series \cite{CoSiTrUl02,BoChLeSaSc07}. To our knowledge, none of these
methods have been proved to provide a complexity which fits in the
bounds obtained in this paper.
\paragraph{Acknowledgment.} This paper is dedicated to Marc Rybowicz,
who passed away in November 2016 \cite{mort-Marc}. The first ideas of
this paper actually came from a collaboration between Marc and the
first author in the beginning of 2012, that led to \cite{PoRy15} as a
first step towards the divide and conquer algorithm presented here.
We also thank Fran\c{c}ois Lemaire for many useful discussions on
dynamic evaluation.


%% file: defpui/defpui.tex
\section{Main definitions and classical algorithms.}
\label{sec:defs}

\subsection{Puiseux series.}
\label{ssec:puiseux}
\input{defpui/cpe}

\subsection{The rational Newton--Puiseux algorithm.}
\label{ssec:NP}
\input{defpui/rnp}

\subsection{Complexity model.}
\label{ssec:comp}
\input{defpui/comp}


%% file: defpui/cpe.tex
We keep notations of Section \ref{sec:intro}. Up to a change of
variable $X \gets X+x_0$, it is sufficient to give definitions and
properties for the case $x_0=0$.  Under the assumption that $p=0$ or
$p>\dy$, the well known Puiseux theorem asserts that the $\dy$ roots
of $F$ (viewed as a univariate polynomial in $Y$) lie in the field of
Puiseux series $\cup_{e\in \mathbb{N}} \algclos{\Ki}((X^{1/e}))$. See
\cite{BrKn86, Ei66, Wa50} or most textbooks about algebraic functions
for the $0$ characteristic case. When $p>\dy$, see \cite[Chap. IV,
Sec. 6]{Ch51}.  It happens that these Puiseux series can be grouped
according to the field extension they define. Following Duval
\cite[Theorem 2]{Du89}, we consider decompositions into irreducible
elements:
\begin{eqnarray*}
  F     &  = & \prod_{i=1}^{\rho} F_i \mbox{ with $F_i$ irreducible in } \Ki[[X]][Y] \label{eq:fact1} \\%
  F_i   & = & \prod_{j=1}^{f_i} F_{ij} \mbox{ with $F_{ij}$ irreducible in } \algclos{\Ki}[[X]][Y]  \label{eq:fact2} \\%
  F_{ij} & = & \prod_{k=0}^{e_i-1}  \left(Y-S_{ij}(X^{1/e_i}\zeta_{e_i}^k) \right) \mbox{ with } S_{ij}\in \algclos{\Ki}((X))\label{eq:fact3}
\end{eqnarray*}
with $\zeta_{e_i}\in\algclos{\Ki}$ is a primitive $e_i$-th root of
unity. Primitive roots are chosen so that $\zeta_{ab}^b=\zeta_a$.
\begin{dfn}
  \label{def:cpe}
  The $\dy$ fractional Laurent series
  $S_{ijk}(X)=S_{ij}(X^{1/e_i}\zeta_{e_i}^k)\in\algclos{\Ki}((X^{1/{e_i}}))$
  are called the \emph{classical Puiseux series} of $F$ above
  0. The integer $e_i\in\Ni$ is the \emph{ramification index} of $S_{ij}$.  If
  $S_{ij}\in \algclos{\Ki}[[X^{1/e_i}]]$, we say that $S_{ij}$ \emph{is
    defined at $x_0=0$}.
\end{dfn}
\begin{prop}
  \label{prop:eifi}
  The $\{F_{ij}\}_{1\leq j \leq f_i}$ have coefficients in a degree
  $f_i$ extension $\Ki_i$ of $\Ki$. They are conjugated by the action
  of the Galois group of $\Ki_i/\Ki$.  We call $\Ki_i$ the
  \emph{residue field} of any Puiseux series of $F_i$ and $f_i$ its
  \emph{residual degree}. We have the relation
  $\sum_{i=1}^{\rho} e_i\,f_i =\dy$.
\end{prop}
\begin{proof}
  First claim is \cite[Section 1]{Du89}. Second one is
  e.g. \cite[Chapter 4, Section 1]{Ch51}.
\end{proof}
This leads to the definition of rational Puiseux expansions (classical
Puiseux series can be constructed from a system of rational Puiseux
expansions - see e.g. \cite[Section 2]{PoRy15}):
\begin{dfn}
  \label{dfn:RPE}
  A system of \emph{rational Puiseux expansions} over $\Ki$ ($\Ki$-RPE)
  of $F$ above 0 is a set $\{R_i\}_{1\leq i \leq \rho}$ such that:
  \begin{itemize}
  \item $R_i(T) \in \Ki_i((T))^2$;%
  \item
    $R_i(T) = (X_i(T),Y_i(T)) = \left(\gamma_i T^{e_i},
      \sum_{l=n_i}^{\infty}\beta_{il} T^l\right)$,
    with $n_i\in \Zi$, $\gamma_i \neq{} 0$ and
    $\beta_{i,n_i} \neq{} 0$;%
  \item $R_i$ is a parametrisation of $F_i$,
    i.e. $F_i(X_i(T),Y_i(T))=0$;
  \item the parametrisation is irreducible, i.e. $e_i$ is minimal.
  \end{itemize}
  We call $(X_i(0),Y_i(0))$ the \textit{center} of $R_i$. We have
  $Y_i(0)=\infty$ if $n_i<0$, which happens only for non monic
  polynomials.
\end{dfn}
Throughout this paper, we will truncate the powers of $X$ of
polynomials or series. To that purpose, we introduce the following
notation: given $\tau\in\Qi$ and a Puiseux series
$S=\sum_{\alpha\in \Qi} c_{\alpha} X^{\alpha}$, we denote
$\tronc{S}{\tau}=\sum_{\alpha\leq \tau} \alpha X^{\alpha}$ (this sum
having thus a finite number of terms).  We generalize this notation to
polynomials with coefficients in the field of Puiseux series by
applying it coefficient-wise. In particular, if $H\in \Ki[[X]][Y]$ is
defined as $H=\sum_i(\sum_{k\geq 0} \alpha_{ik}X^k)Y^i$, then
$\tronc{H}{\tau} = \sum_i(\sum_{k=0}^{\lfloor \tau
  \rfloor}\alpha_{ik}X^k)Y^i$.
\begin{dfn}\label{dfn:singRPE}
  The \emph{regularity index} $r$ of a Puiseux series $S$ of $F$ with
  ramification index $e$ is the least integer
  $N\geq \min(0,e\,\val(S))$ such that, if
  $\tronc{S}{\frac{N}{e}}=\tronc{S'}{\frac{N}{e}}$ for some Puiseux
  series $S'$ of $F$, then $S=S'$. We call $\tronc{S}{\frac{r}{e}}$
  the \emph{singular part} of $S$ in $F$.
\end{dfn}
Roughly speaking, the regularity index is the number of terms
necessary to ``separate'' a Puiseux series from all the others (with a
special care when $\val(S)<0$). 
\begin{xmp}
  \label{xmp:ri}
  Consider $F_1\in\Fi_{29}[X,Y]$ defined as
  $F_1=\prod_{i=1}^3(Y-S_i(X))+X^{19}Y$ with
  $S_i = X+X^2+X^3+17\,X^4+X^5+X^6+X^7+(-1)^i\,X^{15/2}$,
  $1\leq i \leq 2$ and $S_3=X+X^2+X^3+X^4$. The singular parts of the
  Puiseux series of $F_1$ are precisely the $S_i$, with regularity
  indices respectively $r_1=r_2=15$ and $r_3=4$.
\end{xmp}
Since regularity indices of all
Puiseux series corresponding to the same rational Puiseux expansion
are equal, we define:
\begin{dfn}
  The \emph{singular part} of a rational Puiseux expansion $R_i$ of
  $F$ is the pair
  \[
  \left(\gamma_i T^{e_i}, \Gamma(T)=\sum_{k=n_i}^{r_i} \beta_{ik} T^k\right),
  \]
  where $r_i$ is the regularity index of $R_i$, i.e. the one of any
  Puiseux series associated to $R_i$.
\end{dfn}
Once such a singular part has been computed, the implicit function
theorem ensures us that one can compute the series up to an arbitrary
precision. This can be done in quasi linear time by using a Newton
operator \cite[Corollaries 5.1 and 5.2, page 251]{KuTr78}.

\paragraph{Notations.} In the remaining of the paper, we will denote
$(R_i)_{1\leq i\leq\rho}$ the rational Puiseux expansions of $F$. To
  any $R_i$, we will always associate the following notations:
\begin{itemize}
\item $e_i$, $f_i$  and $r_i$ will respectively be the ramification index,
  the residual degree and the regularity index of $R_i$,
\item we define $v_i\in\Qi$ as $\val(F_Y(S))$ for any Puiseux series
  $S$ associated to $R_i$.
\end{itemize}
Same notations will be used if $S_i$ (or $S_{ijk}$) denotes a Puiseux
series. If we omit any index $i$, we will use the notations $e$, $f$
and $r$ for the three first integers.


%% file: defpui/rnp.tex
Our algorithm in Section \ref{sec:algo} is a variant of the well known
Newton--Puiseux algorithm \cite{Wa50,BrKn86}. We now explain (roughly
speaking) the idea of this algorithm via an example, and then describe
the variant of D. Duval \cite[section 4]{Du89} (we use its
improvements).
\paragraph{Tools and idea of the algorithm.} Let
$F_0(X,Y)=Y^6+Y^5X+5\,Y^4X^3-2\,Y^4X+4\,Y^2X^2+X^5-3\,X^4$ and
consider its Puiseux series computation. From the Puiseux theorem, the
first term of any such series $S(X)$ is $\alpha\,X^{\frac m q}$ for
some $\alpha\in\algclos\Ki$ and $(m,q)\in\Ni^2$. We have
$F_0(X,\alpha\,X^{\frac m q}+\cdots) = \alpha^{6} \, X^{\frac {6\,m}
  q} + \alpha^5 \, X^{\frac {5\,m} q+1} + 5\,\alpha^4\,X^{\frac {4\,m}
  q+3}- 2\,\alpha^4 \,X^{\frac {4\,m} q+1} + 4\,\alpha^2\,X^{\frac
  {2\,m} q+2} + X^5-3\,X^4 + \cdots$.
To get $F_0(X,S(X))=0$, at least two terms of the previous sum must
cancel one another, i.e. $(m,q)$ must be chosen so that two or more of
the exponents coincide. To that purpose, we use the following
definition:
\begin{dfn}
The \emph{support} of $F=\sum_{i,j}\alpha_{ij}X^j\,Y^i$ is the set
  $\{(i,j)\in \Ni^2\,|\, \alpha_{ij}\neq 0\}$.
\end{dfn}
Note that the powers of $Y$ are given by the horizontal axis. The
condition on $(m,q)$ can be translated as: two points of the support
of $F_0$ belong to the same line $\edge{m}{q}{l}$. To increase the
$X$-order of the evaluation, no point must be under this
line. Here we have two such lines, $\edge{}{2}{6}$ and $\edge{}{}{4}$,
that define the Newton polygon of $F_0$:
\begin{dfn} \label{dfn:NP} The \emph{Newton polygon} $\Nc(F)$ of $F$
  is the lower part of the convex hull of its support.
\end{dfn}
We are now considering the choice of $\alpha$ corresponding to
$\edge{}{2}{6}$. We have
$F_2(T^2,\alpha\,T) = (\alpha^6-2\alpha^4+4\alpha^2)\,T^6 - 3\,T^8 +
\alpha^5\,T^7 + (5\alpha^4+1)T^{10} + \dots$, meaning that $\alpha$
must be a non zero root of $P(Z)=Z^6-2\,Z^4+4Z^2$.
Then, to get more terms, we recursively apply this strategy to
the polynomial $F_2(X^2,X\,(Y+\alpha))$. Actually, it is more
interesting to consider a root $\xi=\alpha^2$ of the polynomial
$\phi(Z)=Z^2-2\,Z+4$ (we have $P(Z)=Z^2\,\phi(Z^2)$ and we are obviously
not interested in the root $\alpha=0$), which is the characteristic
polynomial \cite{Du89}:
\begin{dfn}
  If $F=\sum \alpha_{ij}X^jY^i$, then the \emph{characteristic
    polynomial} $\phi_{\Delta}$ of $\Delta\in\Nc(F)$ is
  $\phi_{\Delta}(T) = \sum_{(a,b)\in \Delta} \alpha_{ab}
  T^{\frac{a-a_0}{q}}$
  where $a_0$ is the smallest value such that $(a_0,b_0)$ belongs to
  $\Delta$ for some $b_0$.
\end{dfn}

\paragraph{Description of the algorithm.} We now provide a formal
definition of the \RNP{} algorithm for monic polynomials (see Section
\ref{ssec:proofD3-general} for the non monic case); it uses two sub
algorithms, for each we only provide specifications, and an additional
definition, the \emph{modified} Newton polygon \cite[Definition
6]{PoRy15}. The latter enables \RNP{} to output \emph{precisely} the
singular part. We will not use it in our strategy, except for the
proof of Lemma \ref{lem:Ni-vFi} (see Remark \ref{rem:RNP-ARNP}). For
the sake of completness, we recall it below.

\begin{itemize}
\item If $F=\sum_{i=0}^\dy \alpha_i(X) \, Y^i $, the {\em modified Newton
    polygon} $\Ns(H)$ is constructed as follow: if $\alpha_0=0$
  (resp. $\alpha_0\neq 0$ and the first edge, starting from the left,
  ends at $(1,v_X(\alpha_1)))$, add to $\Nc(F)$ (resp. replace the
  first edge by) a fictitious edge joining the vertical axis to
  $(1,v_X(\alpha_1))$ such that its slope is the largest (negative or
  null) integer less than or equal to the slope of the next edge (see
  Figure \ref{fig:Ns}).%
\item \bezout{}, given $(q,m) \in \Zi^{2}$ with $q>0$, computes
  $(u,v)\in \Zi^2$ s.t. $u\,q-m\,v=1$ and $0\leq v < q$.%
\item \factor{}, given $\Ki$ a field and $\phi$ a univariate
  polynomial over $\Ki$, computes the factorisation of $\phi$ over
  $\Ki$, given as a list of factors and multiplicities.%
\end{itemize}

\begin{algorithm}[ht]
  \nonl{}\TitleOfAlgo{\RNP($F, \Ki, \pi$)\label{algo:RNP}}
  \KwIn{$F\in\Ki[X,Y]$ monic, $\Ki$ a field and $\pi$ the result of
    previous computations ($\pi=(X,Y)$ for the initial call)}%
  \KwOut{A set of singular parts of rational Puiseux expansions above
    $(0,0)$ of $F$ with their base field.}%
  $\Rc\ \gets\ \{\}$\tcp*{results of the algorithm will be grouped in
    $\Rc{}$}%
  \ForEach(\tcp*[f]{we consider only negative
    slopes}){$\Delta\in\Ns(F)$
}{%
    Compute $m,q,l,\phi_\Delta$ associated to $\Delta$\;%
    $(u,v)\ \gets$ \bezout$(m,q)$\;%
    \ForEach{$(\phi,M)$ \In{} \factor$(\phi_\Delta)$}{%
      Take $\xi$ a new symbol satisfying $\phi(\xi)=0$\;%
      $\pi_1=\pi(\xi^v\,X^q,X^m\,(Y+\xi^u))$\;%
      \lIf{$M=1$}{$\Rc\ \gets\ \Rc\ \cup\ \{(\pi_1(T,0),\Ki(\xi))\}$}%
      \Else{%
        $H(X,Y)\ \gets\ F(\xi^vX^q,X^m\,(Y+\xi^u))/X^l$\label{line:RNPShift}\tcp*{Puiseux transform}%
          $\Rc\ \gets\ \Rc\ \cup$ \RNP($H,\Ki(\xi),\pi_1$)\;%
        }%
      }%
    }%
    \Return $\Rc$\;
\end{algorithm}

The key improvement of this rational version is the distribution of
$\xi$ to both $X$ and $Y$ variables (line \ref{line:RNPShift}). This
avoids to work with $\alpha=\xi^{1/q}$ and to introduce any useless
field extension due to ramification (see \cite[Section 4]{Du89}).

\begin{figure}[ht]
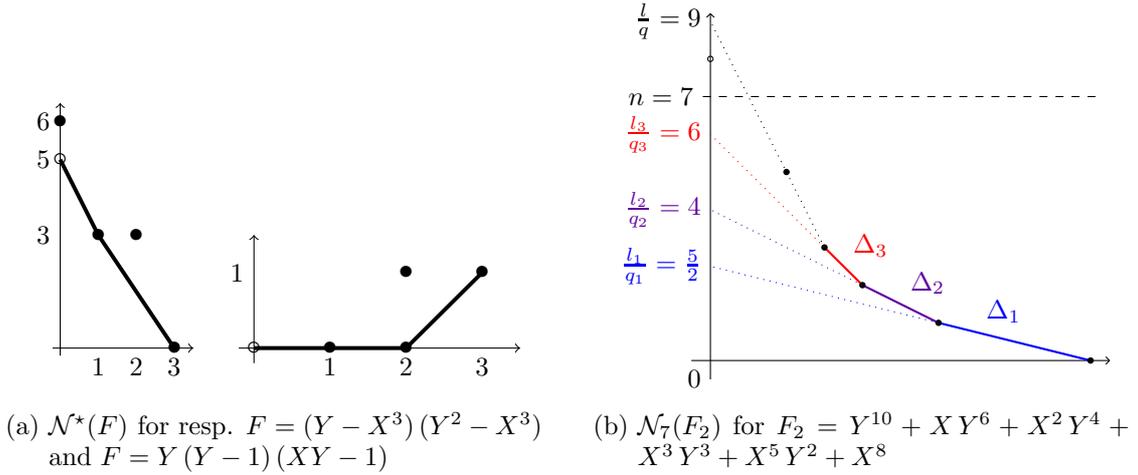

  \begin{subfigure}[b]{.48\linewidth}
    \centering
    \input{figures/fig-Ns.tkz}
    \caption{$\Ns(F)$ for resp.
      $F=(Y-X^3)\,(Y^2-X^3)$ and
      $F=Y\,(Y-1)\,(XY-1)$ \label{fig:Ns}}
  \end{subfigure}%
  \hspace*{.04\linewidth}
  \begin{subfigure}[b]{.48\linewidth}
    \centering
    \input{figures/fig-Nn.tkz}
    \caption{$\Nn[7](F_2)$ for
      $F_2=Y^{10}+X\,Y^6+X^2\,Y^4+X^3\,Y^3+X^5\,Y^2+X^8$\label{fig:Nn}}
  \end{subfigure}
  \caption{The modified and truncated Newton polygons\label{fig:NP}}
\end{figure}

\paragraph{Truncated Newton polygon.} In this paper, we will use low
truncation bounds; in particular, we may truncate some points of the
Newton polygon. In order to certify the correctness of the computed
slopes, we will use the following definition:
\begin{dfn}
  \label{dfn:truncN}
  Given $F\in\Ki[X,Y]$ and $n\in\Ni$, the \emph{$n$-truncated Newton
    polygon} of $F$ is the set $\Nn(F)$ composed of edges $\Delta$ of
  $\Nc(\tronc F n)$ that satisfy $\frac l q \leq n$ if $\Delta$
  belongs to the line $\edge m q l$. In particular, any edge of
  $\Nn(F)$ is an edge of $\Nc(F)$.
\end{dfn}
\begin{xmp}
  Let us consider $F_2=Y^{10}+X\,Y^6+X^2\,Y^4+X^3\,Y^3+X^5\,Y^2+X^8$
  and $n=7$. Figure \ref{fig:Nn} provides the truncated Newton polygon
  of $F_2$ with precision $7$. Here we have $\tronc{F_2}{7}=F_2-X^8$
  and $\Nc(\tronc{F_2}{7})=[(10,0),(6,1),(4,2),(3,3),(2,5)]$. But the
  edge $[(3,3),(2,5)]$ is not part of $\Nn[7](F_2)$, as it belongs to
  $\edge{2}{}{9}$, and that there are points $(i,j)$ so that
  $2\,i+\,j\leq 9$ and $j>7$: from the knowledge of $\tronc{F_2}{7}$,
  we cannot guarantee that $\Nc(F_2)$ contains an edge belonging to
  $\edge{2}{}{9}$. This is indeed wrong here, since
  $\Nc(F_2)=[(10,0),(6,1),(4,2),(3,3),(0,8)]$.
\end{xmp}


%% file: defpui/comp.tex
In this paper, we use two model of computations ; both are RAM models:
the algebraic RAM of Kaltofen \cite[Section 2]{Ka88} and the boolean
one. The latter is considered only for Corollaries \ref{cor:gMC} and
\ref{cor:gLV}, where we just estimate word operations generated by
arithmetic operations in various coefficient fields (assuming for
instance a constant time access to coefficients of polynomials).  For
the arithmetic model, we only count the number of arithmetic
operations (addition, multiplication, division) in our base field
$\Ki$. Most subalgorithms are deterministic; for them,
we consider the worst case. However, computation of primitive elements
uses a probabilistic of Las Vegas type algorithm. Their running times
depend on random choices of element in $\Ki$; hence, we use average
running times, that propagate to our main results.

Our complexity results use the classical notations $\O()$ and $\Ot()$
that respectively hide constant and logarithmic factors. See for
instance \cite[Chapter 25, Section 7]{GaGe13}.
\paragraph{Polynomial multiplication.} We finally recall some
classical complexity results, starting with the multiplication of
univariate polynomials:
\begin{dfn}
  \label{dfn:M}
  A (univariate) \emph{multiplication time} is a map
  $\M : \Ni \rightarrow \Ri$ such that:
  \begin{itemize}
  \item for any ring $\Ai$, polynomials of degree less than $d$ in
    $\Ai[X]$ can be multiplied in at most $\M(d)$ operations
    (multiplication or addition) in $\Ai$;%
  \item for any $0<d\leq d'$, the inequality
    $\M(d)\,d'\leq \M(d')\,d$ 
    holds.
  \end{itemize}
\end{dfn}
\begin{lem}
  \label{lem:M}
  Let $\M$ be a multiplication time. Then we have:
  \begin{enumerate}
  \item $\M(d+d')\geq\M(d)+\M(d')$ for any $d,d'\in\Ni$,
  \item $\M(1)+\M(2)+\cdots+\M(2^{k-1})+\M(2^k)\leq\M(2^{k+1})$ for
    any $k\in\Ni$.
  \end{enumerate}
\end{lem}
\begin{proof}
  The first point is \cite[Exercise 8.33]{GaGe13}. The second one is
  a direct consequence.
\end{proof}
The best known multiplication time gives
$\M(d)\in\O(d\,\log(d)\,\log(\log(d))) \subset \Ot(d)$
\cite{ScSt71,CaKa90}. Note in particular that for this value of $\M()$, we do not have
$\M(d)\,\M(d')\leq\M(d\,d')$ but only
$\M(d)\,\M(d')\leq\M(d\,d')\log(d\,d')$. This is why we use Kronecker
substitution.
\paragraph{Multiplication of multivariate polynomials.} 
Consider two polynomials belonging to $\Ai[Z_1,\cdots,Z_s]$. Denote
$d_i$ a bound for their degrees in $Z_i$. Then, by Kronecker
substitution, they can be multiplied in less than
$\O(\M(2^{s-1}\,d_1\cdots d_s))$ operations in $\Ai$ (it is
straightforward to adapt \cite[Corollary 8.28, page 247]{GaGe13} to
any number of variables). In particular, if $s$ is constant, the
complexity bound is $\O(\M(d_1\cdots d_s))$.
\paragraph{Bivariate polynomials defined over an extension of $\Ki$.}
Given an irreducible polynomial $P\in\Ki[Z]$, we denote
$\Ki_P:=\Ki[Z]/(P(Z))$ and $\dg:=\deg_Z(P)$. In Sections \ref{sec:algo}
and \ref{sec:rnp3}, we multiply two polynomials in $\Ki_P[X,Y]$ as
follows: first perform the polynomial multiplication over $\Ki[X,Y,Z]$
as stated in the previous paragraph; then apply the reduction modulo
$P$ on each coefficient. Denoting $\dx$ (resp. $\dy$) a bound for the
degree in $X$ (resp. $Y$) of the considered polynomials, the total
cost is $\O(\M(\dx\,\dy\,\dg))$ (see \cite[Theorem 9.6, page
261]{GaGe13} for the second point).
\paragraph{Matrix multiplication.} Primitive elements computation are
expressed via the well known $2\leq \omega\leq 3$ exponent (so that
one can multiply two square matrices of size $d$ in less than
$\O(d^\omega)$ operations over the base ring). We have $\omega<2.373$
\cite{Le14}. Note however that our results do not require fast matrix
multiplication: they stand if we take $\omega=3$.

Finally, note that we postpone the discussion concerning the
complexity of operations modulo triangular sets (needed for dynamic
evaluation) in Section \ref{ssec:compD5}.


%% file: arnp/arnp.tex
\section{Refined truncation bounds.}
\label{sec:algo}

We keep notations of Sections \ref{sec:intro}, \ref{ssec:puiseux} and
\ref{ssec:comp} ($\Ki_P$ and $\dg$). Additionally, we assume $F$ to be
monic. The aim of this section is to prove that we can compute at
least half of the Puiseux series of $F$ in less than $\Ot(\dy\,\vRF)$
arithmetic operations, not counting the factorisation of univariate
polynomials. Our algorithms and intermediate results will use the
following notion:

\begin{dfn}\label{dfn:precision}
  We say that $S_0\in\algclos{\Ki}((X^{1/e_0}))$ is a Puiseux series of
  $F$ known with precision $n$ if there exists a Puiseux series $S$ of
  $F$ s.t. $\tronc{S_0}{n} = \tronc{S}{n}$. We say that
  $R_0=(\gamma_0\,T^{e_0},\Gamma_0(T))$ is a \RPE{} of $F$ known with
  precision $n$ if $\tronc{\Gamma_0((X/\gamma_0)^{1/e_0})}{n}$ is a
  Puiseux series of $F$ known with precision $n$.
\end{dfn}

\begin{thm}
  \label{thm:half-Puiseux}
  There exists an algorithm that computes some \RPE{}s
  $R_1,\cdots,R_\lambda$ of $F$ known with precision at least
  $4\,\vRF/\dy$, containing their singular parts, and such that
  $\sum_{i=1}^\lambda e_i\,f_i\geq\frac\dy 2$. Not taking into account
  univariate factorisations, this can be done in an expected
  $\O(\M(\dy\,\vRF)\,\log(\dy)) \subset\Ot(\dy\,\vRF)$ arithmetic
  operations over $\Ki$.
\end{thm}
Algorithm \arnp{} in Section \ref{ssec:algo} will be such an
algorithm. It uses previous improvements by the first author and
M. Rybowicz \cite{PoRy08,PoRy11,PoRy15}, and one additional idea,
namely Idea \ref{id:fact} of Section \ref{sec:intro}.
\subsection{Previous complexity improvements and Idea \ref{id:fact}.}
\label{ssec:previous}
\input{arnp/PoRy}

\subsection{The \arnp{} algorithm.}
\label{ssec:algo}
\input{arnp/algo}

\subsection{Using tight truncations bounds.}
\label{ssec:trunc}
\input{arnp/truncation}

\subsection{Complexity results and proof of Theorem \ref{thm:half-Puiseux}.}
\label{ssec:arnp}
\input{arnp/rnp4}


%% file: arnp/PoRy.tex
\begin{lem}
  \label{lem:onesubs}
  Let $n\in \Ni$, $F\in\Ki_P[X,Y]$ and $\xi\in\Ki_P$ for some
  irreducible $P\in\Ki[Z]$. Denote $\Delta$ an edge of $\Nc(F)$
  belonging to \edge{m}{q}{l}, and $(u,v)$=\bezout$(m,q)$. The Puiseux
  transform $F(\xi^v X^q,X^m (\xi^u +Y))/X^l$ modulo $X^n$ can be
  computed as $n$ univariate polynomial shifts over $\Ki_P$. It takes
  less than $\O(n\,\M(\dy\,\dg))$ operations over $\Ki$.
\end{lem}
\begin{proof}
  This is \cite[Lemma 2, page 210]{PoRy11}; Figure \ref{fig:rnpshift}
  illustrates the idea. Complexity also uses Kronecker substitution.
\end{proof}
Using the \emph{Abhyankar's trick} \cite[Chapter 12]{Ab90}, we reduce
the number of recursive calls of the rational Newton--Puiseux algorithm
from $\vRF$ to $\O(\rho\,\log(\dy))$.
\begin{lem}
  \label{lem:abhyankar}
  Let $F=Y^\dy+\sum_{i=0}^{\dy-1} A_i(X)\,Y^i\in\Ki[X,Y]$ with
  $\dy>1$. If the Newton polygon of $F(X,Y-A_{\dy-1}/\dy)$ has a
  unique edge $(\Delta)\, \edge{m}{q}{l}$ with $q=1$, then
  $\phi_\Delta$ has several roots in $\algclos{\Ki}$.
\end{lem}
In other words, after performing the Tschirnausen transform
$Y\gets{}Y-A_{\dy-1}/\dy$, we are sure to get at least either a branch
separation, a non integer slope, or a non trivial factor of the
characteristic polynomial. This happens at most $\O(\rho\,\log(\dy))$
times.
\begin{xmp} \label{xmp:abhyankar} Let's consider once again the
  polynomial $F_1$ of Example \ref{xmp:ri}. Its Newton polygon has a
  unique edge with integer slope, and the associated characteristic
  polynomial has a unique root. The Abhyankar's trick is applied with
  $\frac 1 3\,A_2 = X+X^2+X^3+2\,X^4+20\,X^5+20\,X^6+20\,X^7$. Then,
  the shifted polynomial has still a unique edge, but its
  characteristic polynomial has two different roots: it separates
  $S_3$ from the two other Puiseux series.%
\end{xmp}
\begin{lem}
  \label{lem:bivshift}
  Let $F=Y^\dy+\sum_{i=0}^{\dy-1} A_i(X)\,Y^i\in\Ki_P[X,Y]$. One can
  compute the truncated shift $\tronc{F(X,Y-A_{\dy-1}/\dy)}{n}$ in less than
  $\O(\M(n\,\dy\,\dg{}))$ operations over $\Ki$.
\end{lem}
\begin{proof}
  From our assumption on the characteristic of $\Ki$, this computation
  can be reduced to bivariate polynomial multiplication via
  \cite[Problem 2.6, page 15]{BiPa94}. The result follows (see Section
  \ref{ssec:comp}).
\end{proof}
In order to provide the monicity assumption of Lemma
\ref{lem:abhyankar}, the well-known Weierstrass preparation theorem
\cite[Chapter 16]{Ab90} is used.
\begin{prop}
  \label{prop:wpt}
  Let $G \in \Ki_P[X,Y]$ not divisible by $X$. There exists unique
  $\Gh$ and $U$ in $\Ki_P[[X]][Y]$ s.t.  $G = \Gh\,U$, with
  $U(0,0)\neq 0$ and $\Gh$ a Weierstrass polynomial of degree
  $\deg_Y(\Gh)=\val[Y](G(0,Y))$. Moreover, \RPE s of $G$ and $\Gh$
  centered at $(0,0)$ are the same.
\end{prop}
The following result provides a complexity bound.
\begin{prop}
  \label{prop:wptalgo} Let $G\in \Ki_P[X,Y]$ as in Proposition
  \ref{prop:wpt} and $n\in\Ni$.  Denote $\Gh$ the Weierstrass
  polynomial of $G$.  There exists an algorithm \wpt{} that computes
  $\tronc{\Gh}{n}$ in less than $\O(\M(n\,\deg_Y(G)\,\dg))$ operations
  in $\Ki$.
\end{prop}
\begin{proof}
  This is \cite[Theorem 15.18, page 451]{GaGe13}, using Kronecker
  substitution for multivariate polynomial multiplication. This
  theorem assumes that $\lc Y G$ is a unit, which is not necessarily
  the case here. However, formul\ae{} in \cite[Algorithm 15.10, pages
  445 and 446]{GaGe13} can still be applied in our context: this is
  exactly \cite[Algorithm Q, page 33]{Mu75}.
\end{proof}
\paragraph{Representation of residue fields.} As explained in
\cite[Section 5.1]{PoRy11}, representing residue fields as multiple
extensions can be costly. Therefore, we need to
compute primitive representations each time we get a characteristic
polynomial $\phi$ with degree $2$ or more. Note that algorithms we use
here are Las-Vegas (this is the only probabilistic part concerning our
results on Puiseux series computation).
\begin{prop}
  \label{prop:primeltKi}
  Let $P\in \Ki[Z]$ and $\phi\in\Ki_P[W]$ be two irreducible
  polynomials of respective degrees $d_P=\deg_Z(P)$ and
  $d_{\phi}=\deg_W(\phi)$. Denote $d=d_P\,d_{\phi}$, and assume that
  there are at least $d^2$ elements in $\Ki$. There exists a Las-Vegas
  algorithm \PrimElt{} that computes an irreducible polynomial
  $P_1\in \Ki[Z]$ with degree $d$ together with an isomorphism
  $\Psi:\Ki_{P,\phi}\simeq \Ki_{P_1}$. It takes an expected
  $\O(d^{\frac{\omega+1}2})$ arithmetic operations plus a constant
  number of irreducibility tests in $\Ki[Z]$ of degree at most
  $d$. Moreover, given $\alpha\in \Ki_{P,\phi}$, one can compute
  $\Psi(\alpha)$ with $\O(d_P\,\M(d))$ operations over $\Ki$.
\end{prop}
\begin{proof}
  See e.g. \cite[Section 2.2]{PoSc13b}; some
  details are in the proof of Proposition \ref{prop:primitive}.
\end{proof}
\begin{rem}
  \label{rem:EnoughElts}
  We do not precisely pay attention to the assumption about the number
  of elements in $\Ki$ in this paper. Note that we will always have
  $d\leq\dy$ in our context. Therefore, if $\Ki$ is a finite field
  without enough elements, it is sufficient to build a degree $2$
  field extension since $p>\dy$.
\end{rem}
\begin{rem}
  \label{rem:omega-rect}
  The above complexity result can actually be expressed as
  $\O(d^{\omega_0})$ where $\frac 3 2 \leq \omega_0 \leq 2$ denotes an
  exponent so that one can multiply a $d \times \sqrt d$ matrix
  and a square $\sqrt d \times \sqrt d$ one with
  $\O(d^{\omega_0})$ operations in $\Ki$. One has $\omega_0<1.667$ from
  \cite{HuPa98}, which is better than the best known bound
  $\frac {\omega+1} 2<1.687$ \cite{Le14}. This however does not
  improve our main results, since we could take $\omega=3$ for our
  results to stand.
\end{rem}
\begin{rem}
  \label{rem:deterministic}
  \cite[Section 4]{HoLe18} provides an almost linear deterministic
  algorithm to compute modulo tower of fields by computing
  ``accelerated towers'' instead of primitive elements. Such a
  strategy would lead to a version of Theorem \ref{thm:half-Puiseux}
  with a deterministic algorithm and a complexity bound
  $\O(\dy^{1+o(1)}\,\delta)$. Their preprint does not however deal with
  dynamic evaluation, so this can not be directly be used in Section
  \ref{sec:D5}, thus in our main results.
\end{rem}

%% file: arnp/algo.tex
We detail the algorithm mentionned in Theorem
\ref{thm:half-Puiseux}. It computes \emph{truncated parametrisations}
of $F$, i.e. maps $\pi=(\gamma\,X^e,\Gamma(X)+\alpha\,X^\tau\,Y)$
s.t. $\pi(T,0)$ is a \RPE{} of $F$ known with precision $\frac \tau e$
(see Definition \ref{dfn:precision}). Except possibly at the first
call, $H$ therein is Weierstrass.

\begin{algorithm}[ht]
  \nonl\TitleOfAlgo{\arnp($H,P,n,\pi$)\label{algo:ARNP}}%
  \KwIn{%
    $P\in \Ki[Z]$ irreducible, $H\in\Ki_P[X,Y]$ separable and monic in
    $Y$ with $d:=\deg_Y(H)>0$, $n\in\Ni$ (truncation order) and $\pi$ the
    current truncated-parametrisation ($P=Z$ and $\pi=(X,Y)$ for the
    initial call).%
  }%
  \KwOut{all RPEs $R_i$ of $H$ s.t. $n-v_i\geq r_i$, 
    with precision $(n-v_i)/e_i \geq r_i/e_i$.}%
  $\Rc\ \gets\ \{\}$ ; $B\ \gets\ A_{d-1}/d$ ; $\pi_1\ \gets\tronc{\pi(X,Y-B)}{n}$ \label{algoARNP:piAbhyankar}\tcp*{$H=\sum_{i=0}^dA_i Y^i$}%
  \leIf{$d=1$}{\Return$\pi_1(T,0)$}{\label{algoARNP:Abhyankarshift} $H_1\gets\tronc{H(X,Y-B)}{n}$}%
  \ForEach(\tcp*[f]{$\Delta$ belongs to $\edge{m}{q}{l}$}){$\Delta$ \In{} $\Nn(H_1)$}{%
    \ForEach{$(\phi, M)$ \In{} $\factor(\Ki_P,\phi_\Delta)$}{%
      \lIf{$\deg_W(\phi)=1$}{$\xi,P_1,H_2,\pi_2=-\phi(Z,0),P,H_1,\pi_1$}%
      \Else{$(P_1,\Psi) \ \gets$ \label{algoARNP:prim} \PrimElt($P,\phi$)\;%
        $\xi,H_2,\pi_2\ \gets \Psi(W),\Psi(H_1),\Psi(\pi_1)$\label{algoARNP:changeRep}\tcp*{$\Psi:\Ki_{P,\phi}\rightarrow \Ki_{P_1}$ isomorphism}}%
      $\pi_3\ \gets \pi_2(\xi^v\,X^q,X^m\,(Y+\xi^u))\mod P_1$\label{algoARNP:piRNPShift}\tcp*{$u,v=\bezout(m,q)$}%
      $H_3\ \gets \tronc{H_2(\xi^v\,X^q,X^m\,(Y+\xi^u))}{n_1}\mod P_1$\tcp*{$n_1=q\,n-l$}\label{algoARNP:updateN}%
      $H_4\ \gets$ \wpt($H_3,n_1$)\label{algoARNP:WPT}\;%
      $\Rc\ \gets\ \Rc\ \cup$ \arnp($H_4, P_1, n_1, \pi_3$)%
    }
  }
  \Return $\Rc$\;%
\end{algorithm}

\begin{rem}
  \label{rem:size-pi}
  We have $\deg_X(\pi)\leq n\,e_i$ for any \RPE{} deduced from
  $\pi$. This is obvious when $\pi$ is defined from Line
  \ref{algoARNP:piAbhyankar}; changing $X$ by $X^q$ on Line
  \ref{algoARNP:piRNPShift} is also straightforward. Also, we have
  $m\leq n\,e_i$, since $\frac m q \leq \frac l q\leq n$ from
  Definition \ref{dfn:truncN}.
\end{rem}

Theorem \ref{thm:half-Puiseux} is an immediate consequence of the
following result, proved in Section \ref{ssec:arnp}.
\begin{prop}
  \label{prop:arnp-correctness}
  \arnp$(F,Z,6\,\vRF/\dy,(X,Y)))$ outputs a set of \RPE{}s.
  Among them is a set $R_1,\cdots,R_\lambda$ known with precision at
  least $4\,\vRF/\dy\ge r_i/e_i$, with $v_i < 2\,\vRF/\dy$ and
  $\sum_{i=1}^\lambda e_i\,f_i \geq \frac\dy 2$.  Not taking into
  account the cost of univariate factorisations, it takes an expected
  $\O(\M(\dy\,\vRF) \, \log(\dy)) \subset\Ot(\dy\,\vRF)$ arithmetic
  operations over $\Ki$.
\end{prop}

\begin{rem}\label{rem:idea-trunc} The key idea is to use tighter
  truncation bounds than in \cite{PoRy11,PoRy15}.  Proposition
  \ref{prop:arnp-correctness} says that $n\in \O(\vRF/\dy)$ is enough
  to get some informations (at least half of the singular parts of
  Puiseux series). This requires a slight modification of
  \cite[Algorithm \texttt{ARNP}]{PoRy15}: $n$ is updated in a
  different way. When there is a transform $X\gets{} X^q$, it must be
  multiplied by $q$; also, it cannot be divided by the degree $t$ of
  the found extension anymore. These points are actually compensated
  by algorithm \wpt{}, which divides the degree in $Y$ by the same
  amount (it eliminates all the conjugates). The size of the input
  polynomial $H$ is thus bounded by $\O(\vRF)$ elements of $\Ki$ (cf
  Section \ref{ssec:arnp}).
\end{rem}

%% file: arnp/truncation.tex
By a carefull study of the \RNP{} algorithm, we get an optimal
truncation bound to compute a \RPE{} of a monic polynomial $F$ with
this algorithm or \arnp{}. From this study, we also deduce an exact
relation between $\vRF$ and this optimal bound. In this section, for
$1\leq i\leq\rho$, we let $\edge{m_{i,h}}{q_{i,h}}{l_{i,h}}$,
$1\leq h\leq g_i$ be the successive edges encountered during the
computation of the expansion $R_i$ with \RNP{}, and denote
\[
N_i:=\sum_{h=1}^{g_i} \frac{l_{i,h}} {q_{i,1}\cdots
  q_{i,h}}.
\]
\begin{lem}
  \label{lem:Ni-vFi}
  For any $1\leq i\leq\rho$, we have $N_i=\frac{r_i}{e_i}+v_i$.
\end{lem}
\begin{proof}
  Denote $R_i=(\gamma_i\,X^{e_i}, \Gamma_i(X,Y))$ with
  $\Gamma_i(X,Y)=\Gamma_{i,0}(X)+X^{r_i}\,Y$. By the definition of the
  Puiseux transformations, we have $(0,1)\in\Nc(G_i)$ for
  \[
  G_i(X,Y):=\frac{F(\gamma_i\,X^{e_i},\Gamma_i(X,Y))}{X^{N_i\,e_i}},
  \]
  i.e.  $\val\left(\partial_Y G_i(X,0)\right)=0$. This is
  $\val(X^{r_i}\,F_Y(\gamma_i\,X^{e_i},\Gamma_{i,0}(X)))=N_i\,e_i$,
  or:
  \[
  N_i=\frac{r_i+\val(F_Y(\gamma_i\,X^{e_i},\Gamma_{i,0}(X)))}{e_i}
  =\frac{r_i}{e_i}+\val(F_Y(X,\Gamma_{i,0}((X/\gamma_i)^{1/e_i})))
  =\frac{r_i}{e_i}+v_i.\qedhere
  \]
\end{proof}
\begin{rem}
  \label{rem:RNP-ARNP}
  This result shows that $N_i$ does not
  depend on the algorithm. Nevertheless, the proof above relies on
  algorithm \RNP{} because it computes \emph{precisely} the singular
  part of all Puiseux series thanks to the \emph{modified} Newton polygon
  \cite[Definition 6]{PoRy15}. The algorithm \arnp{} introduces two
  differences:
  \begin{itemize}
  \item The Abhyankar's trick does not change the value of the $N_i$.
    After applying it, the next value $\frac l q$ is just the addition
    of the $\frac{l_i}{q_i}$ we would have found with \RNP{} (the
    concerned slopes being the sequence of integer slopes that compute
    common terms for \emph{all} Puiseux series, plus the next
    one). See Example \ref{xmp:RNP-ARNP} below.
  \item Not using the modified Newton polygon $\Ns$ can only change
    the last value $\frac l q$ (when the coefficient of
    $X^{r/e}$ is $0$). This has no impact on the proof of Lemma
    \ref{lem:tronc} below.
  \end{itemize}
  In the remaining of this paper, we will define $N_i$ as
  $\frac{r_i}{e_i}+v_i$.
\end{rem}
\begin{xmp}
  \label{xmp:RNP-ARNP}
  Let's assume that $F$ is an irreducible polynomial with Puiseux
  series $S(X)=X^{1/2}+X+X^{3/2}+X^2+X^{9/4}$. The successive values
  for $(l,q)$ are:
  \begin{itemize}
  \item $(4,2)$, $(2,1)$, $(2,1)$, $(2,1)$ and $(2,2)$ with the \RNP{}
    algorithm. We thus get $N = 2 + 1 + 1 + 1 + \frac 1 2 =
    \frac{11}2$.
  \item $(4,2)$, $(14,2)$ with the \arnp{} algorithm (assuming high
    enough truncation). We thus get $N = 2 + \frac 7 2 =
    \frac{11}2$.
  \end{itemize}
\end{xmp}
\input{figures/RNPshift.tkz}
\begin{lem}
  \label{lem:tronc}
  Let $n_0\in\Ni$. To compute the \RPE{} $R_i$ with certified
  precision $n_0\geq \frac{r_i}{e_i}$, it is necessary and sufficient
  to run $\arnp{}$ with truncation bound $n=n_0+v_i$. In particular,
  to ensure the computation of the singular part of $R_i$, it is
  necessary and sufficient to use a truncation bound $n\geq N_i$.
\end{lem}
\begin{proof}
  First note that starting from $H$ known up to $X^n$, the greatest
  $n_1$ so that we can certify
  $H_{\Delta,\xi}:=H(\xi^v\,X^q,X^m\,(Y+\xi^{u}))/X^{l}$ up to
  $X^{n_1}$ is precisely $n_1=q\,n-l$ (see Figure \ref{fig:rnpshift};
  details are in \cite[Proof of Lemma 2]{PoRy11}). This explains the
  truncation update of line \ref{algoARNP:updateN}.

  We now distinguish two cases, according to whether the coefficient
  in $X^{\frac{r_i}{e_i}}$ of any Puiseux series asociated to $R_i$ is
  zero or not. If not, then starting from a truncation bound
  $n=n'+N_i$, we get $n_1=q\,n'+q\,N_i-l$.  By construction,
  $q\,N_i-l$ is precisely the ``$N_i$'' of the associated \RPE{} of
  $H_{\Delta,\xi}$. By induction, we finish at the last call of the
  algorithm associated to the \RPE{} $R_i$ with a truncation bound
  $n=e_i\,n'$. Moreover, we have $\deg_Y(H)=1$ and
  $\pi=(\gamma_i\,X^{e_i},\Gamma_i(X)+\alpha_i\,X^{r_i}\,Y)$. Hence,
  the ouptut $R_i$ is known with precision $n'+\frac{r_i}{e_i}$. We
  conclude thanks to Lemma \ref{lem:Ni-vFi} by taking
  $n'=n_0-\frac{r_i}{e_i}$.

  Finally, if the coefficient in $X^{\frac{r_i}{e_i}}$ of any Puiseux
  series asociated to $R_i$ is zero, we will have
  $\pi=(\gamma_i\,X^{e_i},\Gamma_i(X)+\alpha_i\,X^{\eta_i}\,Y)$ with
  $\eta_i>r_i$. If this is the case, then that means that the at the
  previous step, we already computed some zero coefficients, thus
  losing the same precision $\eta_i-r_i$. This does not change the
  result.
\end{proof}
This proves that 
$N_i$ is an optimal bound to compute the singular
part of the \RPE{} $R_i$. We now bound it.
\begin{lem}
  \label{lem:ri-vi}
  We have $\frac{r_i}{e_i}\leq v_i$.
\end{lem}
\begin{proof}
  This is written in the proof of \cite[Proposition
  5, page 204]{PoRy11}.
\end{proof}
\begin{cor}
  \label{cor:Ni-vi}
  We have $v_i\leq N_i\leq 2\,v_i$.
\end{cor}
\begin{proof}
  Straightforward consequence of Lemmas \ref{lem:Ni-vFi} and
  \ref{lem:ri-vi}.
\end{proof}
We finally deduce global bounds:
\begin{prop} \label{prop:vi-V}
  At least $\frac \dy 2$ Puiseux series
  $S_{i,j,k}$ satisfy $v_i<2\,\vRF/\dy$ and $N_i<4\,\vRF/\dy$.
\end{prop}
\begin{proof}
  Assume the $R_i$ ordered s.t. $v_i\leq v_{i+1}$,
  and let $\lambda$ s.t.
  $\sum_{i=1}^{\lambda-1} e_i\,f_i < \frac \dy 2 \leq
  \sum_{i=1}^\lambda e_i\,f_i$  (i.e.
  $\sum_{i=\lambda+1}^\rho e_i\,f_i \leq \frac \dy 2 <
  \sum_{i=\lambda}^\rho e_i\,f_i$
  by Proposition \ref{prop:eifi}). Then we have
  \[
  \vRF=\sum_{i=1}^\rho v_i\,e_i\,f_i\geq \sum_{i=\lambda}^\rho
  v_i\,e_i\,f_i \geq v_\lambda \sum_{i=\lambda}^\rho e_i\,f_i >
  v_\lambda \frac{\dy}2,
  \]
  the first equality being a resultant property (see
  e.g. \cite[Exercise 6.12]{GaGe13}). Hence, for all $i\le \lambda$,
  we have $v_i\le v_{\lambda} < 2\vRF/\dy$, thus $N_i <4\vRF/\dy$ by
  Corollary \ref{cor:Ni-vi}. The claim follows.
\end{proof}


%% file: arnp/rnp4.tex
\begin{prop}
  \label{prop:arnp-complexity}
  Not taking into account the cost of univariate factorisations,
  running \arnp$(F,Z,n,(X,Y))$ takes an expected
  $\O(\M(n\,\dy^2)\log(\dy))$ operations over $\Ki$.
\end{prop}
\begin{proof}
  Let's consider a function call to \arnp{}($H,P,n_H,\pi$) and denote
  $d_P=\deg_Z(P)$. We distinguish two kind of lines (for both, note
  the bound $n\,\dy\geq n_H\,\deg_Y(H)\,d_P$):
  \begin{enumerate}[(Type 1)]
  \item \label{item-rho} By Lemma \ref{lem:bivshift}, Line
    \ref{algoARNP:Abhyankarshift} takes less than $\O(\M(n\,\dy))$
    operations over $\Ki$. So do Lines \ref{algoARNP:piAbhyankar} and
    \ref{algoARNP:piRNPShift}, by respectively Lemmas
    \ref{lem:bivshift} and \ref{lem:onesubs}, using Remark
    \ref{rem:size-pi} and $e_i\,f_i\leq\dy$.
  \item \label{item-qphi} Lines \ref{algoARNP:updateN} and
    \ref{algoARNP:WPT} are $\O(\M(q\,d_{\phi}\,n\,\dy))$ from
    respectively Lemma \ref{lem:onesubs} and Proposition
    \ref{prop:wptalgo}. By Proposition \ref{prop:primeltKi}, so is
    Line \ref{algoARNP:prim}, while Line \ref{algoARNP:changeRep}
    costs $\O((d_P d_{\phi})^{\frac{\omega+1}{2}})$.
  \end{enumerate}
  From Lemma \ref{lem:abhyankar}, when $q=d_{\phi}=1$, we must have a
  branch separation. Therefore, this happens at most $\rho-1$ times
  (more precisely, the number of pairs $(\Delta,\phi)$ with
  $q=d_{\phi}=1$ while considering all recursive calls is bounded by
  $\rho$). This means that the sum of the costs for these cases is
  less than $\O(\rho\,\M(n\,\dy))\subset\O(\M(n\,\dy^2))$.

  To conclude the proof, we still have to deal with all the cases
  where $q>1$ or $d_{\phi}>1$. In such a case, Type \ref{item-qphi}
  lines are the costly ones. Moreover, we can bound $q$ by $e_i$ and
  $d_P d_{\phi}$ by $f_i$ for any \RPE{} $R_i$ issued from
  $(\Delta,\phi)$. But for each \RPE{} $R_i$, such situation cannot
  happen more than $\log(e_i\,f_i) \leq \log(\dy)$ times (before
  and/or after separation of this branch with other ones). From
  Definition \ref{dfn:M}, that means we can bound the total cost for
  all these cases by  $\O((\M(\sum_{i=1}^\rho e_i\,f_i\,n\,\dy) +
  \sum_{i=1}^{\rho} f_i^{\frac{\omega+1}2})\log(\dy)) \subset
  \O(\M(n\,\dy^2)\log(\dy))$.
\end{proof}
\paragraph*{Proof of Proposition \ref{prop:arnp-correctness}.} As far
as correctness is concerned, we only have to take care of truncations
and the precision of the output: other points are considered in
previous papers of the first author \cite{Po08, PoRy11, PoRy15} (note
also \cite[Section 4.1]{Du89} concerning the construction of the
output). From Lemma \ref{lem:tronc}, a function call
\arnp$(F,Z,6\vRF/\dy,(X,Y)))$ provides (at least) the Puiseux series
satisfying $v_i< 2\vRF/d_Y$ with precision $4\vRF/\dy$ or greater. As
$r_i/e_i \leq v_i$ from Lemma \ref{lem:ri-vi}, their singular parts
are known. Also, from Proposition \ref{prop:vi-V}, we get at least
half of the Puiseux series of $F$.  Complexity is Proposition
\ref{prop:arnp-complexity}. $\hfill\square$


%% file: rnp3/rnp3.tex
\section{A divide and conquer algorithm.}
\label{sec:rnp3}
We keep notations of Sections \ref{sec:intro} and \ref{sec:algo}, and prove in this
section the following result:
\begin{thm}\label{thm:puidV-details}
  Not taking into account the cost of univariate factorisations, there
  exists an algorithm that computes the singular part of all rational
  Puiseux expansions of $F$ above $x_0=0$ in less than
  $\O(\M(\dy\,\vRF)\,\log(\dy\,\vRF)+\M(\dy)\,\log^2(\dy))$ arithmetic
  operations.
\end{thm}
Assuming that $F$ is monic, our strategy can be summarised as follows:
\begin{enumerate}
\item Run \arnp{}($F,Z,6\,\vRF/\dy,(X,Y)$). If this provides all
  \RPE{}s of $F$, we are done. If not, from Section \ref{sec:algo}, we
  get at least half of the Puiseux series of $F$, satisfying $v_i<2\,\vRF/\dy$,
  and known with precision $4\,\vRF/\dy$ or more.
\item From these Puiseux series, construct the associated irreducible
  factors and their product $G$ with precision $4\,\vRF/\dy$ ; cf
  Section \ref{ssec:norm}. Note that $\deg_Y(G) \geq \dy/2$.
\item Compute its cofactor $H$ by euclidean division modulo
  $X^{4\,\vRF/\dy+1}$.
\item Compute the B\'ezout relation $U\,G+V\,H=X^\kgh \mod X^{\kgh+1}$
  via \cite[Algorithm 1]{MoSc16}. We prove in Section
  \ref{ssec:bezout} that $\kgh\leq 2\,\vRF/\dy$.
\item Using this relation, lift the factorisation
  $F=G\,H \mod X^{4\,\vRF/\dy+1}$ to precision $\vF$ using a variant of
  the Hensel lemma. See Section \ref{ssec:hensel}.
\item Finally, apply the main algorithm recursively on $H$; as the
  degree in $Y$ is at least divided by two each time, this is done at
  most $\log(\dy)$ times, for a total cost only multiplied by
  $2$. This is detailed in Section \ref{ssec:proofD3}.
\end{enumerate}
If $F$ is not monic (this assumption is not part of Theorem
\ref{thm:puidV-details}), first use Hensel lifting to
compute the factor $F_{\infty}$ corresponding to RPEs centered at
$(0,\infty)$ up to precision $X^{\vRF}$. Then, compute the RPEs
of $F_{\infty}$ as ``inverse'' of the RPEs of its reciprocal polynomial
(which is monic by construction). Details are provided in Section
\ref{ssec:proofD3-general}.

\subsection{Computing the norm of a \RPE{}.}
\label{ssec:norm}
\input{rnp3/norm}

\subsection{Lifting order.}
\label{ssec:bezout}
\input{rnp3/bezout}

\subsection{Adaptation of Hensel's lemma to our context.}
\label{ssec:hensel}
\input{rnp3/hensel}

\subsection{The divide and conquer algorithm for monic polynomials.}
\label{ssec:proofD3}
\input{rnp3/proofD3}

\subsection{Dealing with the non monic case: proof of Theorem \ref{thm:puidV-details}.}
\label{ssec:proofD3-general}
\input{rnp3/proofD3-general}


%% file: rnp3/norm.tex
\begin{lem}
  \label{lem:norm}
  Let $R_1,\cdots,R_\lambda$ be a set of $\Ki$-\RPE{}s not centered at
  $(0,\infty)$. For each $R_i$, we denote $(S_{ijk})_{jk}$ its associated
    Puiseux series. Let
  \[
  \nu=\max_{1\leq i\leq\lambda}
  \sum_{\substack{(i',j',k')\\\neq(i,j,k)}}\val(S_{ijk}(X)-S_{i'j'k'}(X))
  \]
  and assume that the $R_i$ are known with precision $n\geq\nu$. Then there
  exists an algorithm \norm{} that computes
  $G\in\Ki[X,Y]$ monic with $\deg_Y(G)=\sum_{i=1}^\lambda e_i\,f_i$,
  $\deg_X(G)=n+\nu$, and such that the \RPE{} of $G$ with precision $n$ are
  precisely the $R_i$. It takes less than
  $\O(\M(n\,\deg_Y(G)^2)\,\log(n\,\deg_Y(G))) \subset{} \Ot(n\,\deg_Y(G)^2)$
  arithmetic operations over $\Ki$.
\end{lem}
\begin{proof}
  Denote $P_i\in\Ki[Z]$ so that
  $R_i=(\gamma_i(Z)\,T^{e_i},\Gamma_i(Z,T))$ is defined over
  $\Ki_{P_i}$. Compute
  \[
  A_i = \prod_{j=0}^{e_i-1}
  \left(Y-\Gamma_i\left(Z,\zeta_{e_i}^j\left(\frac
        X{\gamma_i}\right)^{\frac{1}{e_i}}\right)\right) \mod
  (X^{n+\nu+1},P_i(Z))
  \]
  for $1\leq i\leq\lambda$. As $n\geq \nu$, it takes
  $\O(\M(e_i^2\,n\,f_i)\log(e_i))$ operations in $\Ki$ using a
  sub-product tree. Then, compute
  $G_i=\res[Z](A_i,P_i) \mod X^{n+\nu+1}$. Adapting \cite[Corollary
  11.21, page 332]{GaGe13} to a polynomial with three variables, this
  is $\O(f_i\,\M(n\,e_i\,f_i)\,\log(n\,e_i\,f_i))$. Summing over $i$
  these two operations, this fits into our bound. Finally, compute $G$
  the product of the $G_i$ modulo $X^{n+\nu+1}$ in less than
  $\O(\M(n\,\deg_Y(G))\log(\deg_Y(G)))$ using a sub-product tree
  \cite[Algorithm 10.3, page 297]{GaGe13}. It has the required
  properties.
\end{proof}

%% file: rnp3/bezout.tex
Our algorithm requires to lift some analytic factors $G,H$ of $F$
which are not coprime modulo $(X)$. To this aim, we will generalise
the classical Hensel lifting. The first step is to compute a
generalized B\'ezout relation $UG+VH=X^{\kgh}$ with $\kgh\in\Ni$ minimal.
\begin{dfn}
  Let $G,H\in \Ki[[X]][Y]$ coprime. The \emph{lifting order} of
  $G$ and $H$ is:
  \[
  \kgh(G,H):=\inf\, \{k\in \Ni, \,\, X^k\in (G,H)\}.
  \]
\end{dfn}
We now provide an upper bound for the lifting order
that is sufficient for our purpose.
\begin{prop}
  \label{prop:bound-kgh} If $F=G\cdot{}H$ with $H$ monic, we have
  $\displaystyle \kgh(G,H) \leq \max_{H(S)=0} \val(F_Y(S))$.
\end{prop}
\begin{proof}
  Let $U\,G\,+\,V\,H\,=\,X^\kgh$ in $\Ki[[X]][Y]$, with
  $\kgh=\kgh(G,H)$ minimal. Up to perform the euclidean division of $U$ by
  $H$, we may assume $\deg_Y(U)<\deg_Y(H)=:d$. Moreover, minimality of
  $\kgh$ and monicity of $H$ impose $\val(U)=0$. Denoting
  $S_1,\cdots,S_d$ the Puiseux series of $H$, we have
  $U(S_i)\,G(S_i)\,=\,X^\kgh$ for $1\leq i\leq d$. Using
  interpolation, we get
  \[
    U = \sum_{i=1}^d \frac{X^\kgh}{G(S_i)\,H_Y(S_i)}
    \prod_{j\neq i} (Y-S_j) = \sum_{i=1}^d \frac{X^\kgh}{F_Y(S_i)}
    \prod_{j\neq i} (Y-S_j).
  \]
  As $\val(U)=0$ and $\val(S_j)\ge 0$ ($H$ is monic), we have
  $\displaystyle{} \kgh\leq \max_{1\le i\leq d}\val(F_Y(S_i))$.
\end{proof}
\begin{cor}
  \label{cor:bound-kgh} Assume that $F\in\Ki[[X]][Y]$ is a non
  irreducible monic polynomial. Then there exists a
  factorisation $F=G\,H$ in $\Ki[[X]][Y]$ such that
  $\kgh(G,H)\leq 2\,\vRF/\dy$.
\end{cor}
\begin{proof}
  From Proposition \ref{prop:vi-V}, there exist $\lambda\ge 1$ \RPE{}
  $R_1,\cdots,R_\lambda$ of $F$ such that $v_i < 2\,\vRF/\dy$ for all
  $i\leq\lambda$. Considering $H=\prod_{i=1}^\lambda F_i$ and
  $G=\prod_{i=\lambda+1}^\rho F_i$ (with $F_i$ the analytic factor
  associated to $R_i$ - see Section \ref{ssec:puiseux}), we are
  done from Proposition \ref{prop:bound-kgh}.
\end{proof}
The relation $U\,G+V\,H=X^\kgh\mod X^{\kgh+1}$ can be computed in
$\O(\M(\dy\,\kgh)\,\log(\kgh) + \M(\dy)\,\kgh \log(\dy))$
\cite[Corollary 1]{MoSc16}. This is $\O(\M(\vRF)\,\log(\vRF))$ for
$(G,H$) of Corollary \ref{cor:bound-kgh}

%% file: rnp3/hensel.tex
We generalise the classical Hensel lemma \cite[section 15.4]{GaGe13}
when polynomials are not coprime modulo $X$. First, the following
algorithm ``double the precision'' of the lifting: given $F$, $G$,
$H$, $U$, $V\in \Ki[X,Y]$ with $H$ monic in $Y$, and $n_0,\kgh\in\Ni$
satisfying
\begin{itemize}
\item $F=G\,H\mod X^{n_0}$ with $n_0>2\,\kgh$,
\item $U\,G+V\,H=X^\kgh\mod X^{n_0-\kgh}$ with $\deg_Y(U)<\deg_Y(H)$,
$\deg_Y(V)<\deg_Y(G)$,
\end{itemize}
it outputs polynomials
$\Gt$, $\Ht$, $\Ut$, $\Vt\in\Ki[X,Y]$ with $\Ht$ monic in $Y$ such that:
\begin{itemize}
\item $F=\Gt\,\Ht\mod X^{2\,(n_0-\kgh)}$, with
  $\Gt=G\mod X^{n_0-\kgh}$ and $\Ht=H\mod X^{n_0-\kgh}$,
\item $\Ut\,\Gt+\Vt\,\Ht=X^\kgh\mod X^{2\,n_0-3\,\kgh}$ ;
  $\deg_Y(\Vt)<\deg_Y(\Gt)$, $\deg_Y(\Ut)<\deg_Y(\Ht)$.%
\end{itemize}

In what follows, \quorem{} denotes
the classical euclidean division algorithm.
\begin{algorithm}[ht]
  \nonl\TitleOfAlgo{\onestep($F,G,H,U,V,n_0,\kgh$)\label{algo:OneStep}}
  $\alpha\assign{}X^{-\kgh}(F-G\cdot H) \mod X^{2\,(n_0-\kgh)}$\;%
  $Q,R\assign{}\quorem_Y(U\cdot\alpha,H) \mod X^{2\,(n_0-\kgh)}$\;%
  $\Gt\assign{}G+\alpha\cdot{}V+Q\cdot G\mod X^{2\,(n_0-\kgh)}$\;%
  $\Ht\assign{}H+R\mod X^{2\,(n_0-\kgh)}$\;%
  $\beta\assign{}X^{-\kgh}(U\cdot{}\Gt+V\cdot\Ht)-1\mod
  X^{2\,n_0-3\,\kgh}$\;%
  $S,T\assign{}\quorem_Y(U\cdot\beta,\Ht) \mod X^{2\,(n_0-\kgh)}$\;%
  $\Ut\assign{}U-T\mod X^{2\,n_0-3\,\kgh}$\;%
  $\Vt\assign{}V-\beta\cdot{}V-S\cdot\Gt\mod X^{2\,n_0-3\,\kgh}$\;%
  \Return{$\Ht$, $\Gt$, $\Ut$, $\Vt$}%
\end{algorithm}

\begin{lem}
  \label{lem:onestep}
  Algorithm \onestep{} is correct; it runs in $\O(\M(n_0\,\dy))$
  operations in $\Ki$.
\end{lem}
\begin{proof}
  From $\alpha\equiv{}0\mod X^{n_0-\kgh}$ (thus
  $Q\equiv{}0\mod X^{n_0-\kgh}$ and $R\equiv{}0\mod X^{n_0-\kgh}$ from
  \cite[Lemma 15.9, (ii), page 445]{GaGe13}) and
  $U\cdot{}G + V\cdot{}H - X^\kgh \equiv{} 0 \mod X^{n_0-\kgh}$, we
  have $\Gt\equiv{} G\mod X^{n_0-\kgh}$,
  $\Ht\equiv{} H\mod X^{n_0-\kgh}$ and
  \begin{eqnarray*}
    F - \Gt\cdot\Ht & \equiv &  F - (G+\alpha\cdot{}V+Q\cdot{G})\cdot
                               (H + \alpha\cdot{}U - Q\cdot{}H)\\%
                    & \equiv & \alpha(X^\kgh - V\cdot{}H - U\cdot{}G) -
                               \alpha^2\cdot{}U\cdot{}V - 
                               Q\cdot\alpha (U\cdot{}G-V\cdot{}H) + 
                               Q^2\cdot{}G\cdot{}H\\%
                    & \equiv & 0\mod X^{2\,(n_0-\kgh)}.
  \end{eqnarray*}
  From $\beta \equiv 0 \mod X^{n_0-2\,\kgh}$ and
  $U\cdot\Gt + V\cdot\Ht - X^\kgh \equiv 0 \mod X^{n_0-\kgh}$, we have:
  \begin{eqnarray*}
    \Ut\cdot\Gt + \Vt\cdot\Ht - X^\kgh & \equiv & (U-U\cdot\beta+S\cdot\Ht)\cdot\Gt+
                                               (V-\beta\cdot{}V-S\cdot\Gt)\cdot\Ht-X^\kgh\\%
                                    & \equiv & U\cdot\Gt + V\cdot\Ht - X^\kgh -
                                               \beta\cdot(U\cdot\Gt + V\cdot\Ht)\\%
                                    & \equiv & \beta\cdot(X^\kgh - U\cdot\Gt -
                                               V\cdot\Ht) \equiv 0 \mod X^{2\,n_0-3\,\kgh}.
  \end{eqnarray*}
  Conditions on the degrees in $Y$ for $\Ht$ and $\Ut$ are obvious
  (thus is the monicity of $\Ht$). The complexity result is similar to
  \cite[Theorem 9.6, page 261]{GaGe13}.
\end{proof}
Assuming we start from a relation $F=G\,H \mod X^{2\,\kgh+1}$ with a
B\'ezout relation $U\,G+V\,H=X^\kgh\mod X^{\kgh+1}$, we thus can
iterate this algorithm up to the wanted precision:
\begin{lem}
  \label{lem:hensel}
  Given $F,G,H$ as in the input of algorithm \onestep{} with
  $n_0=2\,\kgh+1$, there exists an algorithm \hensel{} that computes
  polynomials $(\Gt,\Ht)$ as in the output of \onestep{} for any
  precision $n\in\Ni$, additionally satisfying:
  \begin{itemize}
  \item $\Gt=G\mod X^{\kgh+1}$, $\Ht=H\mod X^{\kgh+1}$ and
    $F=\Gt\cdot{}\Ht\mod X^{n+2\,\kgh}$;%
  \item if there are $G^\star$, $H^\star\in\Ki[X,Y]$
    satisfying $F=G^\star\cdot{}H^\star\mod X^{n+2\,\kgh}$, then $\Gt=G^\star
    \mod X^n$ and $\Ht=H^\star \mod X^n$.%
  \end{itemize}
  It takes less than
  $\O(\M(n\,\dy)+\M(\kgh\,\dy)\,\log(\kgh\,\dy))$ operations in $\Ki$.
\end{lem}
\begin{proof}
  The algorithm runs as follows:
  \begin{enumerate}
  \item Compute $U,V\in\Ki[X,Y]$ s.t.
    $U\cdot{}G+V\cdot{}H=X^\kgh \mod X^{\kgh+1}$ \cite[Algorithm 1]{MoSc16}.
  \item Double the value $n_0-2\,\kgh$ at each call of \onestep{}, until
    $n_0-2\,\kgh\geq n+\kgh$.
  \end{enumerate}
  Correctness and complexity follow Lemma \ref{lem:onestep} (using
  \cite[Corollary 1]{MoSc16} for the computation of $U$ and
  $V$). Finally, uniqueness of the result is an adaptation of
  \cite[Theorem 15.14, page 448]{GaGe13} (this works because we take a
  precision satisfying $n_0-2\,\kgh\geq n+\kgh$).
\end{proof}
\begin{rem}
  \label{rem:hensel}
  Note that if $G(0,Y)$ and $H(0,Y)$ are coprime, then $\kgh=0$ and
  this result is the classical Hensel lemma.
\end{rem}


%% file: rnp3/proofD3.tex
We provide our divide and conquer algorithm. Algorithm \quo{}
outputs the quotient of the euclidean division in $\Ki[[X]][Y]$ modulo
a power of $X$, and $\#\Rc$ is the cardinal of $\Rc$.

\begin{algorithm}[ht]
  \nonl\TitleOfAlgo{\wrnp($F,n$)\label{algo:RNP}}%
  \KwIn{%
    $F\in\Ki[X,Y]$, separable and monic in $Y$ ; $n \in \Ni$ ``big enough''.%
  }%
  \KwOut{%
    the singular part (at least) of all the RPEs
    of $F$ above $x_0=0$.%
  }%
  \leIf{$\dy<6$}{\Return\arnp($F,Z,n,(X,Y)$)\label{algowrnp:linedy6}}{$\eta\gets{}6\,n/\dy$}%
  $\Rc\;\gets\;\arnp(F,Z,\eta,(X,Y))$\label{algowrnp:lineARNP} \; %
  Keep in $\Rc$ the RPEs with $v_i<\eta/3$\tcp*{known with precision $\ge 2\eta/3$}%
  \lIf{$\#\Rc=\dy$}{\Return
    $\Rc$}%
  $G\;\gets\;\norm(\Rc,2\eta/3)$\;\label{algowrnp:lineNorm}%
  $H\;\gets\;\quo(F,G,2\eta/3)$\;\label{algowrnp:lineQuo}%
  $G,H\;\gets\;\hensel(F,G,H,n)$\;\label{algowrnp:lineHens}%
  \Return $\Rc\;\cup\;\wrnp(H,n)$\;\label{algowrnp:lineRec}%
\end{algorithm}

\begin{prop}
  \label{prop:Wrnp3}
  If $n\geq\vRF$, \wrnp$(F,n)$ returns the correct ouput in an
  expected $\O(\M(\dy\,n)\log(\dy\,n))$ operations in $\Ki$, plus the
  cost of univariate factorisations.
\end{prop}
\begin{proof}
  We start with correctness. As precision
  $n\geq\vRF$ is sufficient to compute the singular parts of all
  Puiseux series via algorithm \arnp{}, the output is correct
  when $\dy<6$. When $\dy\geq 6$, Line \ref{algowrnp:lineARNP}
  provides a set of \RPE{}s $(R_i)_{1\leq i\leq\lambda}$ known
  with precision $\eta-v_i$ by Lemma \ref{lem:tronc}. At line
  \ref{algowrnp:lineNorm}, we keep in $\Rc$ the \RPE{}s $R_i$ such
  that $v_i<\eta/3$; they are thus known with precision at least
  $2\,\eta/3$. Also, we have $\deg_Y(G)\geq\dy/2\geq\deg_Y(H)$ from
  Proposition \ref{prop:vi-V}. Finally, input of the \hensel{}
  algorithm is correct since $\kgh(G,H)$ is less than $\eta/3$ by
  Proposition \ref{prop:bound-kgh} and we know the factorisation
  $F=G\cdot{}H\mod X^{2\,\eta/3+1}$.

  We now focus on complexity. By Proposition
  \ref{prop:arnp-complexity}, Lines \ref{algowrnp:linedy6} ($\dy$ is constant) and \ref{algowrnp:lineARNP} are respectively
  $O(\M(n))$ and $\O(\M(n\,\dy)\,\log(\dy))$. Lines
  \ref{algowrnp:lineNorm}, \ref{algowrnp:lineQuo} and
  \ref{algowrnp:lineHens} take respectively
  $\O(\M(n\,\dy)\,\log(n\,\dy))$, $\O(\M(n\,\dy))$ and
  $\O(\M(n\,\dy)+\M(\vF)\,\log(\vF))$ by respectively Lemma
  \ref{lem:norm}, division via Newton iteration \cite[Theorem 9.4]{GaGe13} and Lemma \ref{lem:hensel}. This
   fits into our result (remember $n\geq\delta$). Finally, as
  $\deg_Y(H)\leq\dy/2$, we conclude from Lemma \ref{lem:M}.
\end{proof}


%% file: rnp3/proofD3-general.tex
\begin{prop}\label{prop:multihensel}
  There exists an algorithm \monic{} that given $n\in \Ni$ and
  $F\in \Ki[X,Y]$ primitive in $Y$, returns $u\in \Ki[X]$ and
  $F_0,F_{\infty}\in \Ki[X,Y]$ s.t.  $F=u\,F_0 F_{\infty}\mod X^n$,
  with $F_0$ monic in $Y$, $F_\infty(0,Y)=1$, and $u(0)\ne 0$ with
  $\O(\M(n\,\dy))$ operations over $\Ki$.
\end{prop}
\begin{proof}
  This is \cite[Algorithm Q, page 33]{Mu75} (see the proof of
  Proposition \ref{prop:wptalgo}).
\end{proof}
We can now give our main algorithm \rnp{}. It computes the singular
part of all RPEs of $F$ above $x_0=0$, including those centered at
$(0,\infty)$. This algorithm, called with parameters $(F,\vRF)$ is the
algorithm mentioned in Theorem \ref{thm:puidV-details}.

\begin{algorithm}[ht]
  \nonl\TitleOfAlgo{\rnp($F,n$)\label{algo:WRNP}}%
  \KwIn{%
    $F\in\Ki[X,Y]$, separable in $Y$ and $n \in \Ni$ ``big enough''.%
  }%
  \KwOut{%
    the singular part (at least) of all the RPEs of $F$ above $x_0=0$%
  }%
  $(u,F_0,F_\infty) \gets $ \monic($F,n$)\;%
  $\tilde F_\infty\gets Y^{\deg_Y(F_\infty)}F_{\infty}(X,1/Y)$\;%
  $\Rc_\infty\gets$ \wrnp($\tilde F_\infty,n$)\label{algornp:wrnp2}\;%
  Inverse the second element of each $R\in\Rc_\infty$\;%
  \Return{} \wrnp($F_0,n$)$\;\cup\;\Rc_\infty$\;%
\end{algorithm}

The proof of Theorem \ref{thm:puidV-details} follows immediately from
the following proposition:
\begin{prop}\label{prop:rnp}
  Not taking into acount the cost of univariate factorisations,
  \rnp{$(F,\vRF)$} returns the correct output with an expected
  $\O(\M(\dy\,\vRF)\,\log(\dy\,\vRF))$ arithmetic operations.
\end{prop}
There is one delicate point in the proof of Proposition
\ref{prop:rnp}: we need to invert the RPEs of $\tilde F_\infty$ and it
is not clear that the truncation bound $n=\vRF$ is sufficient for
recovering in such a way the singular part of the RPEs of $F_\infty$
(see also Remark \ref{rem:vRF} below). We will need the two following
results:
\begin{lem}
  \label{lem:v1/s}
  Consider two distinct Puiseux series $S$ and $S_0$. Then we have
  \[
  \val\left(\frac 1 S - \frac 1 {S_0}\right) = \val(S - S_0) - \val(S) -
  \val(S_0).
  \]
\end{lem}
\begin{proof}
  If $\val(S)\neq \val(S_0)$, one can assume $\val(S)<\val(S_0)$,
  i.e. $\val(S-S_0)=\val(S)$ and
  $\val\left(\frac 1 S - \frac 1 {S_0}\right) = \val(S_0)$. If
  $\val(S)=\val(S_0)=\alpha$, then $\frac {X^\alpha} S\mod X^n$ is
  uniquely determined from $X^{-\alpha}\,S\mod X^n$ (same for
  $S_0$). If $s=\val(S-S_0)-\alpha$,  we have
  $X^{-\alpha}\,S=X^{-\alpha}\,S_0\mod X^s$, i.e.
  $\frac {X^\alpha}S = \frac {X^\alpha}{S_0}\mod X^s$ and
  $\alpha+\val\left(\frac 1 S - \frac 1 {S_0}\right) \geq s$.
  Similarly, denoting
  $s_0=\alpha+\val\left(\frac 1 S-\frac 1 {S_0}\right)$, we get
  $\val(S-S_0)-\alpha\geq s_0$, concluding the proof.
\end{proof}
 \begin{prop}\label{prop:prec-infty}
   Let $F_\infty\in\Ki[X,Y]$ with $F_\infty(0,Y)=1$ and denote
   $\tilde F_\infty$ its reciprocal polynomial according to $Y$. For
   each RPE $R_i=(\lambda_i X^{e_i},\Gamma_i)$ of $F_{\infty}$, denote
   $s_i:=\val(\Gamma_i)$ $($so $s_i<0)$, $r_i$ its regularity index
   and $\tilde R_i$ the associated RPE of $\tilde F_\infty$. The
   function call \wrnp$(\tilde F_\infty,\vRF[F_\infty])$ computes each
   RPE $\tilde R_i$ with precision at least $\frac{r_i-2\,s_i}{e_i}$.
 \end{prop}
\begin{proof}
  Denote $d=\deg_Y(F_\infty)$, $v=\val(\lc Y {F_{\infty}}$,
  $S_1,\cdots,S_d$ the Puiseux series of $F_\infty$ and $S_{k_i}$ one
  of them associated to the RPE $R_i$ of $F_\infty$ we
  are considering. Then $\frac{s_i}{e_i}=\val(S_{k_i})$ and
  $\val(S_{k_i}-S_j)\leq \frac{r_i}{e_i}$ for $j\neq k_i$ by
  definition of $r_i$.  Let $i_0$ satisfying
  $\displaystyle \val(S_{k_i}-S_{i_0})=\max_{j\neq
    k_i}\val(S_{k_i}-S_j)$
  (several values of $i_0$ are possible). We distinguish three cases:
  \begin{enumerate}
  \item \label{enum:eq} $\val(S_{k_i})=\val(S_{i_0})$; then either
    $\val(S_{k_i}-S_{i_0})=\frac{r_i}{e_i}$, or $e_{i_0}=q\,e_i$ with
    $q>1$. In the latter case, there exist $q$ conjugates Puiseux
    series $S_{i_0}^{[0]},\cdots,S_{i_0}^{[q-1]}$ of $S_{i_0}$ such
    that
    $\val(S_{k_i}-S_{i_0})=\val\left(S_{k_i}-S_{i_0}^{[j]}\right)$,
    thus
    $\sum_{j=0}^{q-1} \val\left(S_{k_i}-S_{i_0}^{[j]}\right) \geq
    \frac{r_i}{e_i}$;
    see \cite[Case 3 in Proof of Proposition 5, pages 204 and
    205]{PoRy11} for details.
  \item \label{enum:gt} $\val(S_{k_i})>\val(S_{i_0})$. Then $r_i=s_i$
    and $\val(S_{k_i})>\val(S_j)$ for $j\neq k_i$ by definition of
    $i_0$. This means $\frac \vinf d\geq-\frac{s_i}{e_i}$ as
    $\vinf=\sum_{k=1}^d -\frac{s_k}{e_k}$ from \cite[Lemma 1, page
    198]{PoRy11}.
  \item \label{enum:lt} $\val(S_{k_i})<\val(S_{i_0})$. Then
    $\val(S_{k_i}-S_{i_0})=s_i=r_i$. We can also assume that
    $\val(S_j)\neq\val(S_{k_i})$ for all $j\neq k_i$: if
    $\val(S_j)=\val(S_{k_i})$, then
    $\val(S_{k_i}-S_j)=\val(S_{k_i}-S_{i_0})$ and one could use
    $i_0=j$ and deal with it as Case \ref{enum:eq}.
  \end{enumerate}
  We can now prove Proposition \ref{prop:prec-infty}. First, for Case
  \ref{enum:eq}, knowing $\frac 1 {S_{k_i}}$ with precision
  $\tilde v_i:=\val\left(\partial_Y \tilde F_\infty\left(\frac
      1{S_{k_i}}\right)\right)$
  is sufficient: from Lemma \ref{lem:v1/s}, we have
  $\tilde v_i=\sum_{j\neq i} \val\left(\frac 1{S_{k_i}} - \frac
    1{S_j}\right)=\sum_{j\neq i}\val\left
    (S_{k_i}-S_j\right)-\val(S_{k_i})-\val(S_j)$.
  As either
  $\val(S_{k_i}-S_{i_0})-\val(S_{k_i})-\val(S_{i_0})=\frac{r_i-2\,s_i}{e_i}$
  or
  $\sum_{j=0}^{q-1}\val\left(S_{k_i}-S_{i_0}^{[j]}\right)-\val(S_{k_i})-\val(S_{i_0}^{[j]})
  \geq \frac{r_i-2\,s_i}{e_i}$, we are done.

  Then, concerning Case \ref{enum:gt}, from Proposition
  \ref{prop:Wrnp3} and Lemma \ref{lem:tronc}, we know the RPE $R_i$
  with precision at least $v_i+\frac v d$, that is at least
  $\frac v d \geq -\frac{s_i}{e_i}$. As $r_i=s_i$, this is at least
  $\frac{r_i-2\,s_i}{e_i}$.

  Finally, Case \ref{enum:lt} requires more attention. Let's first
  assume that $\val(S_{k_i})>\val(S_j)$ for some $j\neq i$; then
  $\val(S_{k_i}-S_j)-\val(S_{k_i})-\val(S_j)=\val(S_{k_i})=-\frac{s_i}{e_i}$,
  and we are done since $r_i=s_i$. If not, then we have
  $\val(S_{k_i})<\val(S_j)$ for all $j$. This means that $e_i=f_i=1$,
  and that $\Nc(\tilde F_\infty)$ has an edge $[(0,v),(1,v-s_i)]$,
  which is associated to $\tilde R_i$. It is enough to prove that the
  truncation bound used when dealing with this Puiseux series is at
  least $v$. As long as this is not the case, this edge is not
  considered from the definition of $\Nn(H)$; also, at each recursive
  call of \wrnp{} (Line \ref{algowrnp:lineRec}), the value of the
  truncation bound $\eta$ increases (since the degree in $Y$ is at
  least divided by $2$). In the worst case, we end with a degree $1$
  polynomial, thus using $\eta=\vRF[F_\infty]\geq v$. This concludes.
\end{proof}
\begin{proof}[Proof of Proposition \ref{prop:rnp}]
  Let us show that the truncation bound for $\tilde F_\infty$ is
  sufficient for recovering the singular part of the Puiseux series of
  $F_\infty$.  First note that the inversion of the second element is
  done as follows: consider
  $\tilde R_i(T)=(\gamma_i\,T^{e_i},\tilde
  \Gamma_i(T)=\sum_{k=0}^{\tau_i} \tilde \alpha_{i,k}\,T^k)$
  and denote $s_i=-\val[T](\tilde \Gamma_i(T))<0$; we compute the
  inverse of $T^{s_i}\,\tilde \Gamma_i(T)$ (that has a non zero
  constant coefficient) via quadratic Newton iteration \cite[Algorithm
  9.3, page 259]{GaGe13}; it takes less than $\O(\M(\tau_i+s_i))$
  arithmetic operations \cite[Theorem 9.4, page 260]{GaGe13}. In order
  to get the singular part of the corresponding RPE $R_i$ of
  $F_\infty$, we need to know $R_i$ with precision $\frac{r_i}{e_i}$,
  i.e. to know at least $r_i-s_i+1$ terms. It is thus sufficient to
  know $\tilde R_i$ with precision $r_i-2\,s_i$. This holds thanks to
  Proposition \ref{prop:prec-infty}. Correctness and complexity of
  Algorithm \rnp{} then follow straightforwardly from Propositions
  \ref{prop:Wrnp3} and \ref{prop:multihensel}.
\end{proof}
\begin{rem}
  \label{rem:vRF}
  Note that precision $\val(\disc F)$ is not always enough to get the
  singular part of the Puiseux series centered at $(0,\infty)$, as
  shows the following example. Consider
  $F_3(X,Y)=1+X\,Y^{d-1}+X^{d+1}\,Y^d$. The singular parts of its RPEs
  are $(T,\frac{-1}{T^d})$ and $(-T^{d-1},\frac 1 T)$. Its reciprocal
  polynomial is $\tilde F_3=Y^d+X\,Y+X^{d+1}$, with RPE's singular
  parts $(T,0)$ and $(-T^{d-1},T)$. Here we have $\val (\disc F)=d$,
  and $\tronc {\tilde F_3} d=Y^d+X\,Y$. The singular parts of
  $\tronc {\tilde F_3} d$ are indeed the same than the one of
  $\tilde F_3$, but we cannot recover the RPE $(T,\frac{-1}{T^d})$ of
  $F$ from the RPE $(T,0)$ of $\tronc{\tilde F_3}d$. Nevertheless, the
  precision $\vRF[F_3]=\val(\lc Y {F_3})+\val (\disc F_3)=2d+1$ is
  sufficient.
\end{rem}
\begin{proof}[Proof of Theorem \ref{thm:puidV-details}.] Compute
  $\vRF$ in less than $\O(\M(\dy\vRF)\log(\dy\vRF))$ operations from
  \cite[Lemma 12]{MoSc16}, then run \rnp{}($F,\vRF$) and conclude from
  Proposition \ref{prop:rnp}.
\end{proof}
\begin{rem}
  Another way to approach the non monic case is the one used in
  \cite{PoRy11}. The idea is to use algorithms \wrnp{} and \hrnp{}
  even when $F$ is not monic. This would change nothing as far as
  these algorithms are concerned, but the proof concerning truncation
  bounds must be adapted:
  \begin{enumerate}
  \item define $s_i:=\min(0,\val(S_i))$,
    $N_i':=N_i-\frac{s_i}{e_i}\,\dy$ and
    $v_i':=v_i-\frac{s_i}{e_i}\,\dy$;
  \item prove $N_i'=\frac{r_i}{e_i}+v_i'$ (use \cite[Figure 3]{PoRy11}
    for possible positive slopes of the initial call);
  \item replace $v_i$ by $v_i'$ and $N_i$ by $N_i'$ in the remaining
    results of Section \ref{ssec:trunc}; proofs use some intermediate
    results of \cite{PoRy11} (in particular, to prove
    $\frac{r_i}{e_i}\leq v_i'$, we need to use some formul\ae{} in the
    proof of \cite[Proposition 5, page 204]{PoRy11}).
  \end{enumerate}
  We chose to consider the monic case separately, since it makes one
  of the main technical results of this paper (namely tight truncation
  bounds) less difficult to apprehend, thus the paper more progressive
  to read.
\end{rem}


%% file: D5/D5.tex
\section{Avoiding univariate factorisation.}
\label{sec:D5}
We proved Theorem \ref{thm:puidV} up to the cost of univariate
factorisations. To conclude the proof, one would additionally need to
prove that the cost of all univariate factorisations computed when
calling Algorithm \arnp{} is in $\Ot(\vRF\,\dy)$. As $\vRF$ can be
small, we would need a univariate factorisation algorithm for a
polynomial in $\Ki[Y]$ of degree at most $d$ with complexity
$\Ot(d)$. Unfortunately, this does not exist. We will solve this point
via Idea \ref{id:D5}; relying on the ``dynamic evaluation'' technique
\cite{DeDiDu85,DaMaScXi05} (also named ``D5 principle'') of Della
Dora, Dicrescenzo and Duval. This provides a way to compute with
algebraic numbers, while avoiding factorisation (replacing it by
\emph{square-free} factorisation). In this context, we will consider
polynomials with coefficients belonging to a direct product of field
extensions of $\Ki$; more precisely to a zero-dimensional \textit{non
  integral} $\Ki$-algebra $\Ki_I=\Ki[\Z]/I$, where $I$ is defined as a
triangular set in $\Ki[\Z]:=\Ki[Z_1,\cdots,Z_s]$. As a consequence,
zero divisors might appear, causing triangular decomposition and
splittings (see Section \ref{ssec:trigsets} for details). Four main
subroutines of the \arnp{} algorithm can lead to a decomposition of
the coefficient ring:
\begin{enumerate}[(i)]
\item \label{enum:NP} computation of Newton polygons,
\item \label{enum:sqrfree} square-free factorisations of characteristic
  polynomials,
\item \label{enum:wpt} subroutine \wpt{}, via the initial gcd
  computation,
\item \label{enum:primelt} computation of primitive elements.
\end{enumerate}
There are two other points that we need to take care of for our main
program:
\begin{enumerate}[(i)]
  \setcounter{enumi}{4}
\item \label{enum:hensel} subroutine \hensel{}, via the initial use of
  \cite[Algorithm 1]{MoSc16};
\item \label{enum:fact0} the initial factorisation of algorithm
  \rnp{} (when computing Puiseux series above \emph{all} critical points).
\end{enumerate}

\begin{rem}\label{rem:dynamic}
  Dynamic evaluation is not the key point of this paper, and has been
  already considered for computing Puiseux series (see
  e.g. \cite{DeDiDu85}). We could have simply said ``split when
  required''. However, keeping quasi-linear algorithms when dealing
  with dynamic evaluation in not an easy task, especially in our
  context where splittings may occur in many various
  subroutines. Hence, we decided to detail all steps and to be precise
  and self-contained about dynamic evaluation in our context. This
  makes this section relatively long and technical, but the reader may
  skip it at a first reading.
\end{rem}

To simplify the comprehension of this section, we will not mention
logarithmic factors in our complexity results, using only the $\Ot$
notation. This section is divided as follows:
\begin{enumerate}
\item We start by recalling a few definitions on triangular sets and
  in particular our notion of \emph{D5 rational Puiseux expansions} in
  Section \ref{ssec:trigsets}.
\item The key point of this section is to deal with these splitting
  with almost linear algorithms; to do so, we mainly rely on
  \cite{DaMaScXi05}. We briefly review in Section \ref{ssec:compD5}
  their results; additionally, we introduce a few algorithms needed
  in our context. In particular, this section details points
  (\ref{enum:primelt}) and (\ref{enum:hensel}) above.
\item Points (\ref{enum:NP}) and (\ref{enum:sqrfree}) above are
  grouped in a unique procedure \NewPol{}, detailed in Section
  \ref{ssec:NewPol}.
\item We provide D5 versions of algorithms \arnp{}, \wrnp{} and
  \rnp{} in Section \ref{ssec:algoD5}.
\item Finally, we prove Theorem \ref{thm:puidV} in Section
  \ref{ssec:proofdV}.
\end{enumerate}

\subsection{Triangular sets and dynamic evaluation.}
\label{ssec:trigsets}
\input{D5/trigsets}

\subsection{Complexity of dynamic evaluation.}
\label{ssec:compD5}
\input{D5/compD5}

\subsection{Computing polygon datas in the D5 context.}
\label{ssec:NewPol}
\input{D5/newpol}

\subsection{Computing half Puiseux series using dynamic evaluation.}
\label{ssec:algoD5}
\input{D5/algo}

\subsection{Proof of Theorem \ref{thm:puidV}.}
\label{ssec:proofdV}
\input{D5/proof-dV}


%% file: D5/trigsets.tex
\begin{dfn}
  \label{dfn:trigset}
  A (monic, autoreduced) \emph{triangular set} of
  $\Ki[Z_1,\cdots,Z_s]$ is a set of polynomials $P_1,\cdots,P_s$ such
  that:
  \begin{itemize}
  \item $P_i\in \Ki[Z_1,\cdots,Z_i]$ is monic in $Z_i$,
  \item $P_i$ is reduced modulo $(P_1,\cdots,P_{i-1})$,
  \item the ideal $(P_1,\cdots,P_s)$ of $\Ki[\Z]$ is radical.
  \end{itemize}
  We abusively call an ideal $I\subset \Ki[\Z]$ a triangular set if it
  can be generated by a triangular set $(P_1,\ldots,P_s)$. We denote
  by $\Ki_I$ the quotient ring $\Ki[\Z]/(I)$.
\end{dfn}
Note that this defines a zero-dimensional lexicographic Gr\"obner
basis for the order $Z_1<\cdots{}<Z_s$ with a triangular structure.
Such a product of fields contains zero divisor:
\begin{dfn}
  \label{dfn:regular}
  We say that a non-zero element $\alpha\in\Ki_I$ is \emph{regular} if
  it is not a zero divisor. We say that a polynomial or a
  parametrisation defined over $\Ki_I$ is regular if all its non
  zero coefficients are regular.
\end{dfn}
\paragraph{Triangular decomposition.} Given a zero divisor $\alpha$ of
$\Ki_I$, one can divide $I$ as $I=I_0\cap{}I_1$ with $I_0+I_1=(1)$,
$\alpha\mod I_0=0$ and $\alpha\mod I_1$ is invertible. Moreover, both
ideals $I_0$ and $I_1$ can be represented by triangular sets of
$\Ki[\Z]$.
\begin{dfn}
  \label{dfn:trig-decomp}
  A \emph{triangular decomposition} of an ideal $I$ is
  $I=I_1\cap \cdots \cap I_k$ such that every $I_i$ can be represented
  by a triangular set and $I_i+I_j=(1)$ for $1\leq i\neq j\leq k$.
\end{dfn}
Thanks to the Chinese remainder theorem, the $\Ki$-algebra $\Ki_I$ is
isomorphic to $\Ki_{I_1}\oplus\cdots\oplus\Ki_{I_k}$ for any
triangular decomposition of $I$. We extend this isomorphism
coefficient wise for any polynomial or series defined above $\Ki_I$.
\begin{dfn}
  \label{dfn:splitting}
  Consider any polynomial or series defined above $\Ki_I$. We
  define its \emph{splitting} according to a triangular decomposition
  $I=I_1\cap\cdots\cap I_k$ the application of the above isomorphism
  coefficient-wise.
\end{dfn}
A key point (as far complexity is concerned) is the concept of \emph{non
  critical} triangular decompositions. We recall \cite[Definitions 1.5
and 1.6]{DaMaScXi05}:
\begin{dfn}
  \label{dfn:coprime}
  Two polynomials $a,b \in \Ki_I[X]$ are said \emph{coprime} if the
  ideal $(a,b) \subset \Ki_I[X]$ is equal to $(1)$.
\end{dfn}
\begin{dfn}
  \label{dfn:noncrit}
  Let $(P_1,\cdots,P_s)$ and $(\tilde P_1,\cdots,\tilde P_s)$ be two distinct
  triangular sets. We define the \emph{level} $l$ of this two
  triangular sets to be the least integer such that $P_l\neq \tilde P_l$.  We
  say that these triangular sets are \emph{critical} if $P_l$ and
  $\tilde P_l$ are not coprime in
  $\Ki[Z_1,\cdots,Z_{l-1}]/(P_1,\cdots,P_{l-1})$. A triangular
  decomposition $I=I_1\cap\cdots\cap I_k$ is said \emph{non critical}
  if it has no critical pairs ; otherwise, it is said \emph{critical}.
\end{dfn}
\paragraph{D5 rational Puiseux expansions.}
We conclude this section by defining systems of D5-RPEs over fields
and product of fields. Roughly speaking, a system of D5-RPE over a
perfect field $\Ki$ is a system of RPEs over $\Ki$ grouped together with
respect to some square-free factorisation of the characteristic
polynomials, hence without being necessarily conjugated over $\Ki$.
We have to take care of two main points:
\begin{enumerate}
\item We want correct informations (e.g. regularity indices) before
  fields splittings. To do so, the parametrisations we compute are
  regular (without any zero divisors).
\item We want to recover usual system of RPEs
after fields splittings.
\end{enumerate}
In particular, the computed parametrisations will fit the following
definition:
\begin{dfn}
  \label{dfn:RPED5-0}
  Let $F\in \Ki[X,Y]$ be separable with $\Ki$ a perfect field. A
  system of \emph{D5 rational Puiseux expansions} over $\Ki$ of $F$
  above $0$ is a set $\{R_i\}_i$ such that:
  \begin{itemize}
  \item $R_i\in \Ki_{P_i}((T))^2$ for some square-free polynomial
    $P_i$,
  \item Denoting $P_i=\prod_j P_{ij}$ the univariate factorisation of
    $P_i$ over $\Ki$ and $\{R_{ij}\}_j$ the splitting of $R_i$ according
    to the decomposition $\Ki_{P_i}=\oplus_j \Ki_{P_{ij}}$, then the set
    $\{R_{ij}\}_{i,j}$ is a system of $\Ki$-RPE of $F$ above $0$ (as in
    Definition \ref{dfn:RPE}).
  \end{itemize}
\end{dfn}
In order to deal with \emph{all} critical points in Section
\ref{sec:desing}, we will compute the RPE's of $F$ above a root of a
square-free factor $Q$ of the resultant $R_F$:
\begin{dfn}
  \label{dfn:RPED5-all}
  Let $F\in \Ki_Q[X,Y]$ separable for some $Q\in \Ki[X]$
  square-free. We say that $F$ admits a system of D5-RPE's over
  $\Ki_Q$ above $0$ if there exists parametrisations as in Definition
  \ref{dfn:RPED5-0} that are regular over $\Ki_Q$. Then, a system of
  \emph{D5 rational Puiseux expansions} over $\Ki$ of $F$ above the
  roots of $Q$ is a set $\{Q_i,\Rc_i\}_i$ such that:
  \begin{itemize}
  \item $Q=\prod_i Q_i$,
  \item $\Rc_i$ is a system of D5 RPE's over $\Ki_{Q_i}$ of
    $F(X+z_i,Y)$ above $0$ (in the sense of definition above), where
    $z_i$ is the residue class of $Z$ modulo $Q_i(Z)$.
  \end{itemize}
\end{dfn}


%% file: D5/compD5.tex
\paragraph{Results of \cite{DaMaScXi05}.}
We start by recalling the main results of \cite{DaMaScXi05}, providing
them only with the $\Ot$ notation (i.e. forgetting logarithmic
factors). In particular, we will take $\M(d)\in\Ot(d)$ in the
following. In our paper, we also assume the number of variables
defining triangular sets to be constant (we usually have $s=2$ in our
context).
\begin{dfn}
  \label{dfn:An}
  An \emph{arithmetic time} is a function $I\mapsto \An(I)$ with real
  positive values and defined over all triangular sets in
  $\Ki[Z_1,\cdots,Z_s]$ such that:
\begin{enumerate}
\item For every triangular decomposition $I=I_1\cap\cdots\cap{}I_h$,
 $\An(I_1)+\cdots+\An(I_h)\leq\An(I)$.
\item Any addition or multiplication in $\Ki_I$ can be made in $\An(I)$
  operations over $\Ki$.
\item Given a triangular decomposition $I=I_1\cap\cdots\cap I_h$, one
  can compute a non-critical triangular decomposition of $I$ that
  refines it in less than $\An(I)$ arithmetic operations. We denote
  \rmCP{} such an algorithm.
\item Given $\alpha\in\Ki_I$ and a non-critical triangular
  decomposition $I=I_1\cap\cdots\cap I_h$, one can compute the
  splitting of $\alpha$ in less than $\An(I)$ operations in $\Ki$. We
  denote \Split{} such an algorithm.
\end{enumerate}
\end{dfn}
\begin{thm}
  \label{thm:An}
  Let $I=(P_1,\cdots,P_s)$ be a triangular set, and denote
  $d_i=\deg_{Z_i}(P_i)$. Assuming $s$ to be constant, one can take
  $\An(I) \in \Ot(d_1\cdots{}d_s)$
\end{thm}
\begin{proof}
  This is a special case of the main result of \cite{DaMaScXi05},
  namely Theorem 8.1 therein.
\end{proof}
\begin{prop}
  \label{prop:gcdD5}
  Let $I=(P_1,\cdots,P_s)$, and $A$, $B\in\Ki_I[Y]$ with degrees in $Y$
  less than $d$. Assuming $s$ constant, one can compute the extended
  greatest common divisor of $A$ and $B$ in less than
  $\Ot(d\cdot{}d_1\cdots d_s)$ operations over $\Ki$.
\end{prop}
\begin{proof}
  This is \cite[Proposition 4.1]{DaMaScXi05}.
\end{proof}
\paragraph{Splitting all coefficients of a polynomial.}  In the
remaining of this section, we focus on the case $s=2$, denoting $I=(Q,P)$,
$d_Q=\deg_{Z_1}(Q)$, $d_P=\deg_{Z_2}(P)$ and $d_I=d_Q\,d_P$.
\begin{lem}
  \label{lem:splitH}
  There exists an algorithm \ReducePol{} that, given $H\in\Ki_I[X,Y]$,
  returns a collection $\{(I_k,H_k)_k\}$ such that $I=\cap_k I_k$ is a
  non critical triangular decomposition and the polynomials
  $H_k=H\mod I_k $ are regular over $I_k$. This algorithm performs at
  most $\Ot(\deg_X(H)\,\deg_Y(H)\,d_I)$ operations over $\Ki$.
\end{lem}
\begin{proof}
  As for \cite[Algorithm monic]{DaMaScXi05}, for each
  coefficient of $H$, we split it according to the decomposition of
  $I$ found so far. For each reduced coefficient we get, we test its
  regularity using gcd computation. This gives us a new (possibly
  critical) decomposition of $I$. We run Algorithm \rmCP{} on it. At
  the end, we split $H$ according to the found decomposition.
  Complexity follows from Theorem \ref{thm:An} and Proposition
  \ref{prop:gcdD5}.
\end{proof}
\paragraph{Square-free decomposition above $\Ki_I$.} We say that a
monic polynomial $\phi\in \Ki_I[Y]$ is square-free if the ideal
$I+(\phi)$ is radical. $\phi=\prod_i \phi_i^{n_i}$ is the square-free
factorisation of $\phi$ over $\Ki_I$ if the $\phi_i$ are coprime
square-free polynomials in $\Ki_I[Y]$ and $n_i < n_{i+1}$ for all $i$.
\begin{prop}
  \label{prop:sqrfree}
  Consider $\Ki$ a perfect field with characteristic $p$ and $\phi\in\Ki_I[Y]$
  a monic polynomial of degree $d$. Assuming $p=0$ or $p>d$, there
  exists an algorithm \sqrfree{} that computes a set
  $\{(I_k,(\phi_{k,l},M_{k,l})_l)_k\}$ such that $I=\cap_k I_k$ is a
  non critical triangular decomposition and
  $\phi_k=\prod_l \phi_{k,l}^{M_{k,l}}$ is the square-free
  factorisation of $\phi_k:=\phi\mod I_k$.  It takes less than
  $\Ot(d\,d_I)$ operations over $\Ki$.
\end{prop}
\begin{proof}
  We compute successive gcds and euclidean divisions,
  using Yun's algorithm \cite[Algorithm 14.21, page 395]{GaGe13} (this
  result is in characteristic 0, but works in positive characteristic
  when $p>d$). Each gcd computation is Proposition \ref{prop:gcdD5}.
  We just need to add splitting steps (if needed) 
  in between two calls. The complexity follows by using Proposition
  \ref{prop:gcdD5} in the proof of \cite[Theorem 14.23, page
  396]{GaGe13}, since there are less than $d$ calls to the algorithm
  \rmCP{}.
\end{proof}
\paragraph{Keeping a constant number of variables.} We extend the result
of Proposition \ref{prop:primeltKi} above $\Ki_Q$ for some square-free
polynomial $Q$. This requires additional attention on splittings.
\begin{prop}
  \label{prop:primitive}
  Let $\phi\in\Ki_I[Z_3]$ square-free, $d=d_P\,\deg_{Z_3}(\phi)$. If
  $\Ki$ contains at least $d^2$ elements, there exists a Las-Vegas
  algorithm that computes $(Q_k,P'_k,\psi_k)_k$ satisfying:
  \begin{itemize}
  \item $Q=\prod_k Q_k$,
  \item $P_k'$ is a squarefree polynomial of degree $d$ over
    $\Ki_{Q_k}$,
  \item $\psi_k:\Ki_{I_k}\to \Ki_{I'_k}$ is an isomorphism, where
    $I_k=(Q_k,P,\phi)$ and $I'_k=(Q_k,P'_k)$.
  \end{itemize}
  We call \PrimEltB{} such an algorithm. It takes
  $\Ot(d^{\frac{\omega+1}2}\,d_Q)$ operations over $\Ki$. Given
  $H\in\Ki_{I_k}[X,Y]$, one can compute $\psi_k(H)$ in less than
  $\Ot(\deg_X(H)\,\deg_Y(H)\,d_P\,d\,d_{Q_k})$.
\end{prop}
\begin{proof}
  We follow the Las Vegas algorithm\footnote{here the assumption on
    the number of elements of $\Ki$ is used} given in \cite[Section
  2.2]{PoSc13b}. First, trace computation of the monomial basis takes
  $\O(\M(d\,d_Q))$ operations in $\Ki$ (it is reduced to polynomial
  multiplication thanks to \cite[Proposition 8]{PaSc06}). Then,
  picking a random element $A$, we compute the $2\,d$ traces of powers
  of $A$ by power projection. Methods based on \cite{Sh94} involve
  only polynomial, transposed polynomial and matrix multiplications,
  for a total in $\O(d^{\frac{\omega+1}2}\M(d_Q))$ operations in
  $\Ki$. Finally, our candidate for $P'$ can be deduced via Newton's
  method in $\O(\M(d\,d_Q))$ operations. It remains to test its
  square-freeness, involving gcd over $\Ki_Q$. It takes less than
  $\Ot(d\,d_Q)$ operations over $\Ki$ from Proposition
  \ref{prop:gcdD5}. If a factorisation of $Q$ appears, we run some
  splittings and Theorem \ref{thm:An} concludes.

  To compute $\psi_k$, we first need $d$ additional traces; this is
  once again power projection. Then, one solves a linear system
  defined by a Hankel matrix (see \cite[Proof of Theorem
  5]{Sh94}). This can be done using the algorithm described in
  \cite{BrGuYu80}, that reduces the problem to extended gcd
  computation, thus involves potential decomposition of $Q$. This is
  once again $\Ot(d\,d_Q)$ operations over $\Ki$ (using \rmCP{} if
  needed).

  To conclude, using e.g. Horner's scheme \cite[Section 5.1.3, page
  209]{PoRy11}, rewriting the coefficients of $H\in\Ki_{I_k}[X,Y]$ can
  be done in $\Ot(\deg_X(H)\,\deg_Y(H)\,d_P\,d\,d_{Q_k})$.
\end{proof}
\begin{rem}
  \label{rem:2vars}
  Algorithm \PrimEltB{} keeps the number of
  variables constant (at most two) for the triangular sets we are using
  during the whole algorithm. We do not work with univariate
  triangular sets for two reasons:
  \begin{enumerate}
  \item Computing such triangular set (starting from a bivariate one)
    would lead to a bound in $d_Q^{\frac{\omega+1}2}$, that can
    be $\dt^{\omega+1}$ when the factor $Q$ of the resultant has high
    degree (see Section \ref{sec:desing}). As $\omega>2$, this is too much.
  \item $Q$ (factor of the resultant) and $P$ (residual extension) do
    not provide the same geometrical information.
  \end{enumerate}
\end{rem}
\paragraph{Extending \wpt{} and \hensel{} to the D5
  context.} We conclude this section by providing trivial extension of
the Hensel algorithms: we only need to pay attention to the initial
gcd-computation (for \wpt) or its generalised version of \cite{MoSc16}
(for \hensel).
\begin{prop}
  \label{prop:wpt-D5}
  Let $G\in \Ki_I[X,Y]$ and $n\in\Ni$. There exist an algorithm that
  computes a set $(I_k, \tronc{\Gh_k}{n})$ such that $I=\cap_k I_k$ is
  a non critical decomposition of $I$ and $\Gh_k$ the Weierstrass
  polynomial of $G\mod I_k$. It takes less than
  $\O(\M(n\,\deg_Y(G)\,d_I))$ operations in $\Ki$. We still denote
  \wpt{} such an algorithm.
\end{prop}
\begin{proof}
  First run \ReducePol{} if needed (it is not in our context), getting
  a set $(I_i,G_i')$. Than, for each $i$, use extended Euclidean
  algorithm with parameters $(Y^{M_i},Y^{-M_i}\,G_i(0,Y))$ with
  $M_i=\val[Y](G_i(0,Y)$, getting a decomposition $I_i=\cap_j I_{ij}$
  and associated B\'ezout relations. Compute a non triangular
  decomposition $I=\cap_k I_k$ that refines $\cap_i\cap_j I_{ij}$,
  and reduce $G$ and the B\'ezout relations accordingly. Finally, run
  the Hensel lemma (that does not generate any splitting) on each
  $G_k$, using the associated B\'ezout relation. Complexity follows
  from Lemma \ref{lem:splitH}, Proposition \ref{prop:gcdD5}, Theorem
  \ref{thm:An} and Proposition \ref{prop:wptalgo}.
\end{proof}

\begin{lem}
  \label{lem:xgcdk-D5}
  Given $G$, $H\in\Ki_I[X,Y]$ of degrees in $Y$ bounded by $d$, one
  can compute a set $(I_k,G_k,H_k,U_k,V_k,\eta_k)_k$ such that
  $I=\cap_k I_k$ is a non critical decomposition of $I$,
  $G_k=G\mod I_k$, $H_k=H\mod I_k$ and
  $U_k\cdot{}G_k+V_k\cdot{}H_k=X^{\eta_k}\mod X^{\eta_k+1}$ with
  $\eta_k$ the lifting order of $(G_k,H_k)$.  This takes
  $\Ot(d\,d_I\,\max_k\eta_k)$ operations over $\Ki$.
\end{lem}
\begin{proof}
  As said in the introduction of their paper, \cite[Algorithm
  1]{MoSc16} is ``a suitable adaptation of the half-gcd algorithm'': a
  call to their algorithm uses polynomial multiplication (more
  precisely multiplications of $2\times2$ matrices of univariate
  polynomials), two recursive calls and one computation of the
  ``pseudo-division operator'' $\cal Q$ \cite[Section
  3.1]{MoSc16}, which includes euclidean division, extended
  Euclidean algorithm and Hensel lifting (\cite[Algorithm Q]{Mu75} to
  compute ``normal form'' of polynomials). Whence a
  finite number of call that induce splittings, all
  considered in \cite{DaMaScXi05} (multiplication induces no
  splitting, Euclidean algorithm is the key point of
  \cite{DaMaScXi05}, and \cite[Algorithm Q]{Mu75} induces splitting
  only once, via the extended Euclidean algorithm).
\end{proof}
\begin{prop}
  \label{prop:khensel-D5}
  Let $n\in\Ni$, $F$, $G$, $H\in \Ki_I[X,Y]$ with $H$ monic in $Y$,
  $F=G\,H\mod X^{2\,\eta+1}$ and $\eta\ge \kgh (G,H)$. There exists an
  algorithm that computes a set $\{I_k, G_k, H_k\}_k$ such that
  $I=\cap_k I_k$ is a non critical decomposition of $I$,
  $G_k=G\mod(I_k,X^{\eta_k+1})$, $H_k=H\mod(I_k,X^{\eta_k+1})$ and
  $F\mod I_k=G_k\,H_k\mod X^{n+2\,\eta_k}$, where
  $\eta_k=\kgh(G_k,H_k)$. Moreover, if $G_k^\star$,
  $H_k^\star\in\Ki_{I_k}[X,Y]$ satisfy
  $F\mod I_k=G_k^\star\, H_k^\star\mod X^{n+2\,\eta_k}$, then
  $G_k=G_k^\star \mod X^n$ and $H_k=H_k^\star \mod X^n$. It takes less
  than $\O(\M(n\,\dy\,d_I))$ operations in $\Ki$. We still denote
  \hensel{} such an algorithm.
\end{prop}
\begin{proof}
The D5 adaptation of the \hensel{} algorithm is straightforward:
use Lemma \ref{lem:xgcdk-D5} first, then run \onestep{} as many times as
necessary for each $(I_i, G_i, H_i, U_i, V_i, \kgh_i)$ you get, as in
the proof of Lemma \ref{lem:hensel}.
\end{proof}


%% file: D5/newpol.tex
To simplify the writing of the \hrnp{} algorithm, we group in
algorithm \NewPol{} below the computation of the Newton polygon and
the square-free decomposition of associated characteristic
polynomials. Given $H\in \Ki_I[X,Y]$ known with precision $n$, it
returns a list $\{(I_i,H_i,\Delta_{ij},\phi_{ijk})\}_k$ such that:
\begin{itemize}
\item $I=\cap I_i$ is a non critical triangular decomposition;
\item $H_i:=H\mod I_i$ is regular;
\item $\Nn(H_i)=\{\Delta_{ij}\}_j$;
\item $\prod_k\phi_{ijk}^{M_{ijk}}$ is the square-free factorisation
  of $\phi_{\Delta_{ij}}$.
\end{itemize}

\begin{algorithm}[ht]
  \nonl\TitleOfAlgo{\NewPol($H, I, n$)\label{algo:NewPol}}%
  \KwIn{%
    $I$ a bivariate triangular set and $H\in\Ki_I[X,Y]$ known modulo
    $X^{n+1}$. We assume $n>0$ and $\deg_Y(H)>0$.%
  }%
  \KwOut{%
    A list $\{(I_i,H_i,\Delta_{ij},\phi_{ijk},M_{ijk})\}$ as explained
    above.%
  }%
  \ForEach{$(H_i,I_i)$ \In{} \ReducePol$(H,I)$}{%
    $\{\Delta_{ij}\}_{j=1,\ldots,s} \, \gets \, \Nn(H_i)$
    \tcp*{$H_i$ is regular}%
    \For{$j=1,\ldots,s$}{%
      $\{I^l_i,\phi_{ijk}^l,M_{ijk}^l\} \gets $
      \sqrfree$(\phi_{\Delta_{ij}},I_{i})$%
    }%
  }%
  $\{I'_h\}_h\gets$ \rmCP{}($\{I_i^l\}_{i,l}$)\; $\{H'_h\}_h \gets$
  \Split{}$(H_i,\{I_i^l\}_{i,l},\{I'_h\}_h)$\;%
  \ForEach{$i,j,k$}{%
    $\{\phi'_{mjk}\}_{mjk}\gets$
    \Split{}$(\phi_{ijk}^l,\{I^l_i\}_l,\{I'_h\}_h)$\tcp*{{\small taking the
      right subset} $\{I'_h\}_h$}%
  }%
  \Return
  $\{(I'_h,H'_h,\Delta_{i(m)j},\phi'_{mjk})\}_{m,j,k}$\tcp*{$i(m): m\mapsto$
    correct $i$}%
\end{algorithm}
\begin{prop}
  \label{prop:D5NewPol}%
  Algorithm \NewPol{} is correct and takes 
  $\Ot(\deg_X(H)\,\deg_Y(H)\,d_I)$ operations in $\Ki$.
\end{prop}
\begin{proof}
  Exacteness and complexity follow from Proposition \ref{prop:sqrfree}
  and Theorem \ref{thm:An}, using
  $\sum_{j,k,l} \deg(\phi_{ijk}^l)\le d_Y(H)$ for all $i$ and
  $\sum_{i,l} \deg(I^l_i) = \sum_h \deg(I'_h) = \sum_i \deg(I_i) =
  d_I$.
\end{proof}


%% file: D5/algo.tex
In order to compute also the RPEs of $F$ above the roots of any
squarefree factor $Q$ of the resultant, we are led to consider
$I=(Q,P)$ instead of $P$ as an input for \hrnp{}, the D5 variant of
\arnp{}. More precisely, the input is a set $H$, $I$, $n$, $\pi$ such
that:
\begin{itemize}
\item $I=(Q,P)$ is a bivariate triangular set over $\Ki$ ($P=Z_2$
  initially, $Q=Z_1$ admitted);%
\item $H\in\Ki_I[X,Y]$ separable, monic in $Y$, with
  $d:=\deg_Y(H)>0$;%
\item $n\in\Ni$ is the truncation order we will use for the powers of
  $X$ during the algorithm;%
\item $\pi$ the current truncated parametrisation ($\pi=(X,Y)$ for the
  initial call).%
\end{itemize}
The output is a set $\{I_i,\Rc_i\}_i$ such that:
\begin{itemize}
\item $I=\cap_i I_i$ is a non critial decomposition,
\item $\Rc_i=\{R_{ij}\}$ is a set of D5-RPE's of $H_i:=H\mod I_i$
  satisfying $n - v_{ij} \ge r_{ij}$ and given with
  precision at least $(n-v_{ij})/e_{ij}\geq r_{ij}/e_{ij}\geq 0$,
\end{itemize}
where we let $v_{ij}:=\val\left(\partial_Y H_i(S)\right)$ for any
Puiseux series $S$ associated to $R_{ij}$.  We refer to the field
version \arnp{} for all notations which are not specified here.

\begin{algorithm}[ht]
  \nonl\TitleOfAlgo{\hrnp($H,I,n,\pi$)\label{algo:HRNP}}
  $B\ \gets\ A_{d-1}/d$ ; $\pi'\ \gets\tronc{\pi(X,Y-B)}{n}$ \label{algoARNP:piAbhyankar}\tcp*{$H=\sum_{i=0}^dA_i Y^i$}%
  \leIf{$d=1$}{\Return$(I,\pi'(T,0))$}{\label{algoARNP:Abhyankarshift} $H'\gets\tronc{H(X,Y-B)}{n}$}%
  $(I_i, H_i, \Delta_i, \phi_i)_i\gets$ \NewPol$(H',I,n)$\;%
  $\{\pi_i\}_i\gets$ \Split{}($\pi',\{I_i\}_i$)\tcp*{taking only once each different $I_i$}%
  \ForAll{$i$}{%
    \lIf{$\deg(\phi_i)=1$}{$\xi_{i1},I_{i1},H_{i1},\pi_{i1}=-\phi_i(0),I_i,H_i,\pi_i$}%
    \Else{%
      $\{I_{ij},\Psi_{ij}\}_j\gets$\label{algoHRNP:changeRep}\PrimEltB($I_i,\phi_i$)\;%
      $\{H_{ij}'\}_j\gets{}$ \Split($H_i,\{I_{ij}\}_j$) ; $\{\pi_{ij}'\}_j\gets{}$ \Split($\pi_i,\{I_{ij}\}_j$)\;%
      \lForAll{$j$}{$\xi_{ij},H_{ij},\pi_{ij}\gets\Psi_{ij}(Z),\Psi_{ij}(H_{ij}'),\Psi_{ij}(\pi_{ij}')$\label{algoHRNP:prim}}%
    }%
    \ForAll(\tcp*[f]{$\Delta_i$ belongs to $\edge{m_i}{q_i}{l_i}$ ; $u_i,v_i=\bezout(m_i,q_i)$}){$j$}{%
      $\pi_{ij}''\gets\pi_{ij}(\xi_{ij}^{v_i}\,X^{q_i},X^{m_i}\,(Y+\xi_{ij}^{u_i}))\mod I_{ij}$\label{algoHRNP:piRNPShift}\;%
      $H_{ij}''\ \gets \tronc{H_{ij}(\xi_{ij}^{v_i}\,X^{q_i},X^{m_i}\,(Y+\xi_{ij}^{u_i}))}{n_i}\mod I_{ij}$ \label{algoHRNP:updateN}\tcp*{$n_i=q_i\,n-l_i$}%
      $\{(I_{ijk},H_{ijk})\}\gets$ \wpt($H_{ij}'',n_i$)\label{algoHRNP:WPT}\;%
      $\pi_{ijk}\ \gets\ \Split(\pi_{ij}'',\{I_{ijk}\}_{ijk})$\;
      \lForAll{$k$}{$\{I_{ijkl},\Rc_{ijkl}\}_l\gets$ \hrnp($H_{ijk}, I_{ijk}, n_i,\pi_{ijk}$)\label{algoHRNP:recRNPShift}}%
    }%
  }%
  $\Rc\gets\{\}$ ; $\{I_h'\}_h\gets$ \rmCP($\{I_{ijkl}\}_{ijkl}$)\;%
  \ForAll(\tcp*[f]{taking the subset of $\{I_h'\}_h$ refining $I_{ijkl}$}){$i,j,k,l$}{%
    $\Rc\gets\Rc\;\cup\;$\Split($\Rc_{ijkl},\{I_h'\}_h$)%
  }%
  \Return $\Rc$\tcp*{each element of $\Rc$ coupled to their associated $I'_h$}%
\end{algorithm}

\begin{prop}\label{prop:D5hrnp}
  Let $Q\in \Ki[Z]$ be square-free and $F\in \Ki_Q[X,Y]$ be monic and
  separable in $Y$. The function call \hrnp$(F,(Q,Z),n,(X,Y))$ returns
  a correct answer in an expected $\Ot(d_Q\,n\,\dy^2)$ operations over
  $\Ki$.
\end{prop}
\begin{proof}
  Just adapt the proof of Proposition \ref{prop:arnp-complexity} to
  the D5 context, using Propositions \ref{prop:primitive}, \ref{prop:wpt-D5}  and
  \ref{prop:D5NewPol}, together with Theorem \ref{thm:An}.
\end{proof}


%% file: D5/proof-dV.tex
We finally conclude the proof of Theorem \ref{thm:puidV},
providing the D5 variants of algorithms \wrnp{} and \rnp{}, namely
algorithms \wdrnp{} and \drnp{} below.
\paragraph{The monic case.}
As in Section \ref{sec:rnp3}, we begin with the monic case. Therein,
we assume that the \hensel{} algorithm is a D5 version, as explained
in Section \ref{ssec:compD5}. Also, we recall that $v_{ij}$ denotes
$\val\left(\partial_Y H_i(S)\right)$ for any Puiseux series $S$
associated to $R_{ij}$.

\begin{algorithm}[ht]
  \nonl\TitleOfAlgo{\wdrnp($F,Q,n$)\label{algo:RNP3-D5}}%
  \KwIn{%
    $Q\in \Ki[Z]$ square-free, $F\in\Ki_Q[X,Y]$ separable and monic
    in $Y$, and $n \in \Ni$.%
  }%
  \KwOut{%
  $\{(Q_i,\Rc_i)\}_i$, with $Q=\prod Q_i$ and $\Rc_i$ a system of singular parts of
    D5-RPEs of $F\mod Q_i$ above $0$.%
  }%
  $\eta\gets\min(n,6\,n/\dy)$ ; $\Rc\gets\{\}$\;%
  $\{I_i,\Rc_i\}_i\gets$
  \hrnp($F,(Q,Z_2),\eta,\pi$)\label{D5algornp:lineARNP}\tcp*{$I_i=(Q_i,Z_2)$}%
  $\{F_i\}_i\gets$ \Split($F,\{Q_i\}_i$)\;%
  \ForAll{$i$}{%
    Keep in $\Rc_i$ the $R_{ij}$ such that $v_{ij} <\eta/3$\tcp*{known with precision $\ge 2\eta/3$}%
    \lIf{$\# \Rc_i=\dy$}{%
      $\Rc\gets\Rc\cup\;\{Q_i,\Rc_i\}$ ; \Continue{}%
    }%
    $G_i\;\gets\;\norm(\Rc_i,2\eta/3)$\label{D5algornp:norm}\;%
    $H_i\;\gets\;\quo(F_i,G_i,2\eta/3)$\tcp*{no splitting since $G_i$
      is monic}%
    $\{Q_{ij},G_{ij},H_{ij}\}_j\gets$ \hensel($F_i,G_i,H_i,n$)\label{D5algornp:hensel}\;%
    \lForAll{$j$}{%
      $\{(Q_{ijk},\Rc_{ijk})\}_k\gets$
      \wdrnp($H_{ij},Q_{ij},n,\pi$)%
    }%
    $\{\Rc'_{ijk}\}\gets$ \Split($\Rc_i,\{Q_{ijk}\}_{j,k}$)\;%
    $\Rc\gets\Rc\cup\{(Q_{ijk},\Rc_{ijk}\;\cup\;\Rc_{ijk}')_{j,k}\}$\;%
  }%
  \Return $\Rc$
\end{algorithm}

Recall the notations $R_F=\res(F,F_Y)$ and $\vRF=\val \, (R_F)$. We
obtain:
\begin{prop}
  \label{prop:D5rnp3W}
  Assuming that $n\geq\vRF$ and that the trailing coefficient of $R_F$
  is not a zero divisor in $\Ki_Q$, a function call \wdrnp($F,Q,n$)
  returns a correct answer in an expected $\Ot(d_Q\,\dy\,n)$ operations
  over $\Ki$.
\end{prop}
\begin{proof}
  The assumption on the trailing coefficient of the resultant of $F$
  is needed only to ensure that the truncation bound $\vRF$ is enough
  over all factors of $Q$. Otherwise, this is just an adaptation of
  the proof of Proposition \ref{prop:Wrnp3} to the D5 context, using
  Propositions \ref{prop:khensel-D5} and \ref{prop:D5hrnp}, together
  with Theorem \ref{thm:An} once again (subroutine \quo{} is used only
  with monic polynomials, and the remaining operations do not include
  any division).
\end{proof}
\paragraph{The general case.}
Algorithm \drnp{} below computes a system of singular part (at least)
of D5-RPEs of a primitive polynomial $F$ above the roots of any
square-free factor $Q$ of its resultant $R_F$.  We follow the same
strategy as in Algorithm \rnp{}, but we take care of triangular
decompositions due to division by zero divisors.  In particular, we
assume that algorithm \monic{} is a D5 version (it contains one call
to the extended Euclidean algorithm). Also, inversion of the RPEs of
$\tilde F_{\infty}$ can lead to some splittings (while inverting the
trailing coefficient of the series). However, we do not detail these
further splittings for readibility.

\begin{algorithm}[ht]
  \nonl\TitleOfAlgo{\drnp($F,Q,n$)\label{algo:RNP3}}%
  \KwIn{%
    $Q\in\Ki[Z_1]$ square-free, $F\in\Ki[X,Y]$ separable in $Y$ with
    $\dy>0$, and $n\in\Ni$ big enough.%
  }%
  \KwOut{%
    A system of singular parts (at least) of
    D5-RPEs of $F$ above the roots of $Q$.%
  }%
  $\Rc\gets \{\}$\ ;\ $\tilde F\gets\tronc{F(X+Z_1,Y)\mod Q}{n}$\tcp*{thus
    $\tilde F\in \Ki_Q[X,Y]$}%
  $\{Q_i,F_{i,0},F_{i,\infty}\}_i\gets$ \monic($\tilde F,n$)\;%
  \ForAll{$i$}{%
    $\tilde F_{i,\infty}\gets Y^{\deg_Y(F_{i,\infty})}F_{i,\infty}(X,1/Y)$\;%
    $\{Q_{ij},\Rc_{ij}\}_j\gets$
    \wdrnp($F_{i,0},Q_i,n$)\label{D5-algornp3:wrnp}\;%
    $\{Q_{ik}',\Rc_{ik}'\}_k\gets$
    \wdrnp($\tilde F_{i,\infty},Q_i,n$)\label{D5-algornp3:wrnp2}\;%
    \ForAll{$k$}{%
      Inverse the second element of each $R\in\Rc_{ik}'$\;
      Split $\{Q_{ik}',\Rc_{ik}'\}$ if required\;%
    }%
    $\{Q_{il}^{''}\}_l\gets$ \rmCP($\{Q_{ij}\}_j\cup
    \{Q_{ik}'\}_k$)\;%
    \lForAll{$k,j$}{ $\Rc\gets\Rc\;\cup$
      \Split($\Rc_{ij}$,$\{Q_{il}^{''}\}_l$) $\cup$
      \Split($\Rc_{ik}'$,$\{Q_{il}^{''}\}_l$)
    }%
  }%
  \Return{} $\Rc$\tcp*{elements of $\Rc$ with the same
    $Q''_{il}$ grouped together}%
\end{algorithm}

\begin{prop}\label{prop:D5rnp3}
  Assuming that $Q$ is a square-free factor of $R_F$ with multiplicity
  $n_Q\le n$, a function call \drnp($F,Q,n$) returns the correct
  answer in less than $\Ot(d_Q\,\dy\,n)$ operation overs $\Ki$.
\end{prop}
\begin{proof}
  The correctness follows from Propositions \ref{prop:rnp} and
  \ref{prop:D5rnp3W} (the trailing coefficient of the resultant of
  $F_{i,0}$ and $F_{i,\infty}$ is not a zero divisor by
  construction). The complexity follows from Propositions
  \ref{prop:wpt} and \ref{prop:D5rnp3W}, Theorem \ref{thm:An},
  together with the relations $\deg_Y(F_{i,0})+\deg_Y(F_{i,\infty})=\dy$ and
  $\sum_{i}\deg(Q_i)=d_Q$.
\end{proof}
\begin{proof}[Proof of Theorem \ref{thm:puidV}.] The algorithm
  mentionned in Theorem \ref{thm:puidV} is Algorithm \drnp{}, run with
  parameters $Q=Z_1$ and $n=\vRF$, which can be computed via
  \cite[Algorithm 1]{MoSc16} in the aimed bound. Note that as we
  consider the special case $Q=Z_1$, $F$ has coefficients over a field
  and this operation does not involve any dynamic evaluation. The
  function call \drnp($F,Z_1,\vRF)$ fits into the aimed complexity
  thanks to Proposition \ref{prop:D5rnp3}.
\end{proof}

%% file: global/desing.tex
\section{Desingularisation and genus of plane curves
  .}
\label{sec:desing}
It is now straightforward to compute a system of singular parts of D5
rational Puiseux expansions above all critical points. We include the
RPEs of $F$ above $x_0=\infty$, defined as RPEs above $x_0=0$ of the
reciprocal polynomial $\tilde F:=X^{\dx} F(X^{-1},Y)$.
We have $\val(R_{\tilde F})=\dx\,(2\,\dy-1)-\deg(R_F)$.
\begin{dfn}\label{def:desing} Let $F\in \Ki[X,Y]$ be a separable
  polynomial over a field $\Ki$. A D5-desingularisation of $F$ over
  $\Ki$ is a collection
  $\{(\Rc_1,Q_1),\ldots, (\Rc_s,Q_s),\Rc_{\infty}\}$ such that:
  \begin{itemize}
  \item $Q_k\in \Ki[X]$ are pairwise coprime, square-free and satisfy
    $R_F=\prod_{k=1}^s Q_k^{n_k}$, $n_k\in \Ni^*$;
  \item $\Rc_k$ is a system of singular parts (at least) of D5-RPEs of
    $F$ above the roots of $Q_k$;
  \item $\Rc_{\infty}$ is a system of singular parts (at least) of
    D5-RPEs of $F$ above $X=\infty$.
  \end{itemize}
\end{dfn}
Note the following points:
\begin{itemize}
\item we can deduce from a D5-desingularisation of $F$ the singular
  part of the RPE's of $F$ above any root of $R_F$,
\item we allow $n_k=n_l$ for $k\ne l$ (the factorisation
  $R_F=\prod_{k=1}^s Q_k^{n_k}$ is not necessarily a square-free
  factorisation).
\end{itemize}
We obtain the following
algorithm:
\begin{algorithm}[ht]
  \nonl\TitleOfAlgo{\desing($F$)\label{algo:desing}}%
  \KwIn{%
    $F\in\Ki[X,Y]$ separable and primitive in $Y$, with $\dy>0$.%
  }%
  \KwOut{The D5-desingularisation of $F$ over $\Ki$}%
  $\Rc\gets\{\}$\;%
  \lForAll{$(Q,n) \in$ \sqrfree$(R_F)$}{%
    $\Rc\gets\Rc\cup\drnp(F,Q,n)$
  }%
  $n\gets{}\dx\,(2\,\dy-1)-\deg(R_F)$\;%
  \lIf{$n>0$}{$\Rc\gets\Rc\cup\drnp(\tronc{X^{\dx} F(X^{-1},Y)}{n},Z,n)$}%
  \Return{$\Rc$}
\end{algorithm}
\begin{prop}\label{prop:desing}
  Algorithm \desing($F$) works as specified. It takes an expected
  $\Ot(\dx\,\dy^2)$ operations over $\Ki$.
\end{prop}
\begin{proof}
  Correctness is straightforward from Proposition
  \ref{prop:D5rnp3}. The computation of the resultant $R_F$ fits in
  the aimed bound \cite[Corollary 11.21, page 332]{GaGe13}, so is its
  square-free factorisation \cite[Theorem 14.20, page 4]{GaGe13}. The
  complexity is then a consequence of Proposition \ref{prop:D5rnp3},
  using the classical formula
  $\sum_k \deg(Q_k) n_k + \vRF[\tilde F] = \dx\,(2\,\dy-1)$.
\end{proof}

\paragraph{Proof of Theorem \ref{thm:puid3}.} It follows immediately
from Proposition \ref{prop:desing}.\hfill{}$\square{}$

\paragraph{Computing the genus of plane curves: proof of Corollaries
  \ref{cor:g}, \ref{cor:gMC} and \ref{cor:gLV}.}
Let $\{(Q_k,\Rc_k)\}_k$ be a D5-desingularisation of $F$. Since the
D5-RPEs $R_{ki}\in \Rc_k$ are regular by construction, the
ramification indices of all classical Puiseux series (i.e with
coefficients in $\algclos{\Ki}$) determined by $R_{ki}$ are
equal. If $F$ is irreducible over $\algclos{\Ki}$, the
Riemann-Hurwitz formula determines the genus $g$ of the projective
plane curve defined by $F$ as
\[
g=1-\dy+\frac{1}{2}\sum_k
\deg(Q_k)\sum_{i=1}^{\rho_k}f_{ki}(e_{ki}-1),
\]
where $f_{ki}$ and $e_{ki}$ are respectively the residual degrees and
ramification indices of the RPE $R_{ki}$.  This proves Corollary
\ref{cor:g}. Corollaries \ref{cor:gMC} and \ref{cor:gLV}
follow from \cite{PoRy08,PoRy12},
where the authors show that we can reduce $F$ modulo a well chosen
small prime within the given bit complexities.

%% file: facto/facto.tex
\section{Factorisation in $\Ki[[X]][Y]$.}%
\label{sec:facto}
Our aim is to compute the irreducible analytic factors of $F$ in
$\Ki[[X]][Y]$ with precision $X^N$, and to do so in at most
$\Ot(\dy(\vRF+N))$ operations over $\Ki$, plus the cost of one
univariate factorisation of degree at most $\dy$.  The idea is to
first compute a factorisation modulo $X^{\vRF}$, and then to lift this
factorisation thanks to the following result:
\begin{prop}\label{prop:factolift}
  Let $F\in \Ki[[X]][Y]$, separable of degree $d$. Suppose given a
  modular factorisation
  \begin{eqnarray}\label{eq:2kappa}
    F\equiv u F_1\cdots F_k\mod X^n, \quad n> 2\kgh
  \end{eqnarray}
  where $u\in \Ki[[X]]^\times$, for all $i$ either
  $F_i$ or its reciprocal polynomial $\tilde F_i$ is monic, and
  \[
    \kgh=\kgh(F_1,\ldots,F_k):=\max_{I,J}\,\kgh(F_I,F_J),
  \]
  the maximum of the lifting orders being taken over all disjoint
  subsets $I,\,J\subset \{1,\ldots,k\}$, with $F_I=\prod_{i\in I} F_i$.
  Then there exists uniquely determined analytic factors
  $F_1^*,\ldots,F_k^*$ such that $F= u^* F_1^*\cdots F_k^*$, where
  \[
    F_i^*\equiv F_i\mod
    X^{n-\kgh}\,\,\text{ and }\,\,u^*\in \Ki[[X]],\,\,u^*\equiv u\mod
    X^{n-\kgh}.
  \]
  Moreover, starting from $(\ref{eq:2kappa})$, we can compute the
  $F_i^*$ up to an any precision $N\ge n-\kgh$ in $\Ot(dN)$ operations
  over $\Ki$.
\end{prop}
\begin{proof}
  Replace in \cite[Algorithm 15.17]{GaGe13} the use of
  \cite[Algorithm 15.10]{GaGe13} (line 6) by the \onestep{} algorithm,
  and the extented Euclidean algorithm (line 4) by \cite[Algorithm
  1]{MoSc16}. Existence and unicity of the lifting follow from Lemma
  \ref{lem:hensel}. So does complexity.
\end{proof}
\begin{rem}
  This results improves \cite[Lemma 4.1]{Ca86}, where $\kappa$ is
  replaced by $\vRF/2\ge \kgh$. Note that if $\kgh=0$, this is the
  classical multifactor Hensel lifting. Otherwise, note that instead
  of starting from a univariate factorisation, we need to know the
  initial factorisation modulo a higher power of $X$.
\end{rem}
\begin{proof}[Proof of Theorem \ref{thm:anfact}]
  We proceed as follows:
  \begin{enumerate}
  \item Compute $\vRF$ in the aimed bound.
  \item Adapt \drnp{} (called with parameters $F$, $Z$ and $\vRF$):
    \begin{itemize}
    \item Make the \norm{} call (line \ref{D5algornp:norm} of \wdrnp)
      additionally output minimal polynomials of the computed RPEs
      (i.e. the polynomials $G_i$ of Section \ref{ssec:norm});
    \item Replace the \hensel{} call (line \ref{D5algornp:hensel} of
      \wdrnp) by its multi-factor version (i.e. Proposition
      \ref{prop:factolift});
    \item Output the lifted factors instead of the RPEs in \wdrnp{}.
    \end{itemize}
  \item We get factors $\tilde F_i$ known modulo $X^{\vRF+1}$, with
    coefficients in a product of fields $\Ki_{P_i}$ and
    $\sum \deg(P_i)=\sum f_i\le \dy$. Perform the univariate
    factorisation of the $P_i$ and split accordingly the $\tilde F_i$
    to get a factorisation $F= u^* F_1^*\cdots F_k^*$ modulo
    $X^{\vRF}$.
  \item If $n>\vRF$, use Proposition \ref{prop:factolift} to lift this
    factorisation to the required precision.\qedhere{}
  \end{enumerate}
\end{proof}


%% file: conc/conc.tex
\section{Concluding remarks}
\label{sec:conc}

In this paper, we provide worst-case complexity bounds for the local
and global desingularisation which are equivalent (up to a logarithmic factor)
to the computation of respectively the first non-zero coefficient of
the resultant $R_F$ \cite{MoSc16} and the resultant
computation. However, this provides for the moment only a theoretical
algorithm: our algorithm is a combination of many subroutines, and the
implementation of a fast efficient version would require a huge
amount of work, especially due to the dynamic evaluation
part. Moreover, there might be algorithm easier to implement that we
plan to study in future work (see below).

\paragraph{Worst case complexity is sharp.} We begin this section by
providing a family of polynomial for which our complexity bounds are
reached.
\begin{xmp}\label{xmp:hensel-sharp}
  Let $d>3$ be divisible by $2$ and consider $F=Y^d+(Y-X^{d/2})^2$, so
  that $\dx=\dy=\dt=d$. By Hensel's lemma, we have
  $F=G\,H\in \Qi[[X]][Y]$ with $G(0,Y)=Y^{d-2}+1$ and $H(0,Y)=Y^2$. As
  $G(0,Y)$ is square-free, we deduce immediately the singular parts of
  the Puiseux series of $G$ (that is, their constant term here). In
  order to compute the singular parts of (at least half) the Puiseux
  series of $H$ above $0$ using algorithm \rnp3{}, we need to lift
  further the factorisation $F=G\,H\mod X$ up to precision
  $\sigma\in \Theta(\vRF[H]/\deg_Y(H))$, and this precision is sharp
  from Lemma \ref{lem:tronc}. We have $\vRF[H]=\vRF=d^2$ while
  $\deg_Y(H)=2$ is constant. Hence the required precision is in
  $\Theta(d^2)$ and the lifting step costs $\Theta(d^3)=\Theta(D^3)$,
  leading to a cubic complexity in the total degree.
\end{xmp}

\paragraph{Irreducibility test via \arnp{} is $\Omega(\dy\,\vRF)$.} The previous
example shows the sharpness of the divide and conquer strategy. But
even the first step (algorithm \hrnp{}) is sharp, due to the ``blowing
up'' of the Puiseux transform. As a consequence, even for an
irreducible polynomial (where there is no need of the divide and
conquer strategy), complexity of Theorem \ref{thm:puidV} is sharp, as
shows the following example:
\begin{xmp}\label{xmp:irr-sharp}
  Let $d>12$ be divisible by $4$ and consider $F$ to be the minimal
  polynomial of the Puiseux series
  $S(X)=X^{\frac 4 d}+X+X^{\frac{d+1}d}$. We have $\dy=d$, $\vRF=7\,d-13$
  and $\val(F_Y(S))=7-\frac{13}{d}$, and Lemma \ref{lem:tronc} proves
  that we need to consider $\tronc F n$ with
  $n={8-\frac{d}{12}}>\frac\vRF d$, i.e. $F\mod X^8$. We have
  $\Nn(F)=((0,4),(d,0))$ with characteristic polynomial $(T-1)^4$, so
  that $m_1=1$, $q_1=\frac d 4$ and $l_1=d$. We therefore need to
  compute the Puiseux transform
  $G(X,Y)=\tronc{F(X^{\frac d 4},X\,(Y+1))/X^d}{n_1}$ with
  $n_1=\frac d 4\,n-d=d-3$. As $G$ has size
  $d\,n_1\in\Omega(\dy\,\vRF)$, so is the complexity of Lemma
  \ref{lem:onesubs}, thus of Theorem \ref{thm:puidV}.
\end{xmp}
As a consequence, this blowing-up step prevents any Newton--Puiseux
like method for providing an irreducibility test in $\Ki[[X]][Y]$ (or
$\algclos\Ki[[X]][Y]$) in $\Ot(\vRF)$ operations in $\Ki$. We plan to
investigate the approach of Abhyankhar \cite{Ab89} to improve that
point; in particular, we hope such an approach to improve
the practical implementation of the algorithm.

\paragraph{The reverse role strategy.} If we only want Puiseux series
centered at $(0,0)$, we can try to invert the roles played by $X$ and
$Y$: thanks to the inversion formula \cite[Proposition 4.2]{GaGoPo17},
we can recover the singular parts of the Puiseux series of $F$
centered at $(0,0)$ with respect to $Y$ from those of
$\Ft(X,Y)=F(Y,X)$.

Considering Example \ref{xmp:hensel-sharp}, the polynomial
$\Ft\in \Ki[[X]][Y]$ is then Weierstrass of degree $d$. One can
compute $\vRF_{\Ft}=d^2+2\,(d-1)$. Hence, we need a lifting precision
$\tilde{\sigma}\in\Theta(\vRF_{\Ft}/d)=\Theta(d)$ in order to compute
at least half of the Puiseux series of $\Ft$, for a total cost
$\Theta(d^2)$. As $\Ft$ has edge polynomial $(Y^{d/2}-X)^2$, we deduce
that we will in fact separate the singular parts of \emph{all} Puiseux
series of $\Ft$ with precision $\tilde{\sigma}$ - recovering then
those of $F$ by appyling the inversion formula - for a total quadratic
cost $\Theta(d^2)=\Theta(D^2)$ assuming that we may apply the
inversion formula within this bound.

\begin{rem}
  We did not check that applying the inversion formula really fits in
  the aimed bound. This problem is closely related to the computation
  of the reciprocal series of a serie $S\in X\Ki[[X]]^*$, that is the
  series $\tilde{S}\in X\Ki[[X]]^*$ such that $S\circ \tilde{S}=X$. We
  did not pursue further this investigation as Example
  \ref{xmp:milnor} below shows that the reverse role strategy fails in
  general - even assuming fast inversion formula. At minima,
  \cite[Theorem 4.4]{GaGoPo17} shows that computing the
  \emph{characteristic monomials} of the Puiseux series of $F$
  centered at $(0,0)$ assuming that those of $\tilde{F}$ are given
  fits in the aimed bound. This data is of particular importance as it
  allows to compute the topological type of the branches of the germ
  of curve defined by $F$ at $(0,0)$.
\end{rem}

We could hope that there is always such a nice way to choose a
suitable system of local coordinates in order to compute all the
Puiseux series centered at $(0,0)$ - or at least their characteristic
monomials - in less than cubic complexity in the total
degree. Unfortunately, Example \ref{xmp:milnor} below shows that this
is hopeless. With the notations above, we have $\vRF[H]=\mu + n_Y-1$
and $\vRF[\Ht]=\mu+n_X-1$ thanks to \cite[Chapter II, Proposition 1.2,
page 317]{Te73}, with $n_Y:=\deg_Y(H)=\val[Y](F(0,Y))$,
$n_X:=\deg_Y(\Ht)=\val(F(X,0))$ and $\mu:=(F_X,F_Y)_0$ the Milnor
number of the germ of curve defined by $F$ at the origin. Thanks to
the inversion formula, computing (the characteristic monomials of) at
least half of the Puiseux series \textit{centered at $(0,0)$} with
\rnp3{} while allowing the reverse role strategy costs
$\Theta(\mu\min (\dy/n_Y,\dx/n_X))$. Unfortunately, this can be
$\Theta(\dt^3)$:
\begin{xmp} \label{xmp:milnor} Let $d>6$ be divisible by $6$ and let
  $F=(\phi+X^{d/2})^2-\phi^{d/3}$ with $\phi=Y^3-X^2$. So $F$ has
  total degree $\dt=d$. We have
  $F_X=X\left((d\,X^{d/2-1}-4)\,(\phi+X^{d/2})+\frac
    {2d}3\,\phi^{d/3-1}\right)$
  and $F_Y=Y^2\left(6\,(\phi+X^{d/2})-d\,\phi^{d/3-1}\right)$.  As
  $d\geq 12$, we have $(X,6\,(\phi+X^{d/2})-d\,\phi^{d/3-1})_0=3$,
  $(Y,U\,(\phi+X^{d/2})+\frac {2d}3\,\phi^{d/3-1})_0=2$. We also have
  \[
  \begin{array}{rl}
    &\left((3\,d\,X^{d/2-1}-12)\,(\phi+X^{d/2})+2\,d\,\phi^{d/3-1},
      6\,(\phi+X^{d/2})-d\,\phi^{d/3-1}\right)_0\\%
    =&\left(3\,d\,X^{d/2-1}\,(\phi+X^{d/2}),6\,(\phi+X^{d/2})-d\,\phi^{d/3-1}\right)_0\\%
    =&3\,(d/2-1)+\left((\phi+X^{d/2}),\phi^{d/3-1}\right)_0=-3+d^2/2\\%
  \end{array}
  \]
  We finally get $\mu=(F_X,F_Y)_0=6+d^2/2\in\Theta(d^2)$. Since
  $n_Y=6$ and $n_X=4$ we obtain
  $\min
  (\dy\mu/n_Y,\dx\mu/n_X)=d^3/12+d\in\Theta(d^3)=\Theta(D^3)$.
  The reverse role strategy is thus not helpful in that case.
\end{xmp}

More generally the Milnor number is invariant under local
diffeomorphic change of coordinates $\pi:(\Ki^2,0)\to (\Ki^2,0)$. In
Example \ref{xmp:milnor}, we can check that we always have
$\max (n_X(\pi^* F),n_Y(\pi^* F))=\max(n_X,n_Y)$, and - assuming
$\pi$ polynomial - we check further
that we always have
$\min (\deg_X(\pi^* F),\deg_Y(\pi^* F))\ge \min(\dx,\dy)$. Hence, there is
no hope to reduce the polynomial $F$ to a nicer polynomial $G$ having
faster desingularisation at $(0,0)$ (or even faster irreducibility test)
using polynomial diffeomorphism of $(\Ki^2,0)$ before appyling
\rnp3{}. This shows that our complexity results are sharp, and so
independently of the choice of a polynomial local change of
coordinates in $(\Ki^2,0)$.

Note that this example is particularly sparse, but one could for
instance consider the ``dense'' polynomial
$F=Y^{d/3}+\sum_{k=0}^{d/6-1}(\phi+X^{d/2})^2\,\phi^k$ that will lead
to the same conclusion than the one of Example \ref{xmp:milnor}.


%% file: puiseuxd3.bbl
\begin{thebibliography}{10}

\bibitem{Ab90}
S.~Abhyankar.
\newblock {\em Algebraic {G}eometry for {S}cientists and {E}ngineers},
  volume~35 of {\em Mathematical surveys and monographs}.
\newblock Amer. Math. Soc., 1990.

\bibitem{Ab89}
S.~S. Abhyankar.
\newblock Irreducibility criterion for germs of analytic functions of two
  complex variables.
\newblock {\em Advances in Mathematics}, 74(2):190 -- 257, 1989.

\bibitem{AlAtMa17}
P.~Alvandi, M.~Ataei, and M.~Moreno~Maza.
\newblock On the extended hensel construction and its application to the
  computation of limit points.
\newblock In {\em Proceedings of the 2017 ACM on International Symposium on
  Symbolic and Algebraic Computation}, ISSAC '17, pages 13--20, New York, NY,
  USA, 2017. ACM.

\bibitem{BaNaSt13}
J.-D. Bauch, E.~Nart, and H.~Stainsby.
\newblock Complexity of the {OM} factorizations of polynomials over local
  fields.
\newblock {\em LMS Journal of Computation and Mathematics}, 16:139--171, 2013.

\bibitem{BiPa94}
D.~Bini and V.~Y. Pan.
\newblock {\em Polynomial and {M}atrix {C}omputations}, volume~1 of {\em
  Progress in Theoretical Computer Science}.
\newblock Birkh\"{a}user, Saarbr\"ucken, 1994.

\bibitem{BoChLeSaSc07}
A.~Bostan, F.~Chyzak, B.~Salvy, G.~Lecerf, and E.~Schost.
\newblock Differential equations for algebraic functions.
\newblock In {\em Proceedings of the 2007 International Symposium on Symbolic
  and Algebraic Computation}, ISSAC '07, pages 25--32, 2007.

\bibitem{BrGuYu80}
R.~P. Brent, F.~G. Gustavson, and D.~Y. Yun.
\newblock Fast solution of toeplitz systems of equations and computation of
  pad\'e approximants.
\newblock {\em Journal of Algorithms}, 1(3):259 -- 295, 1980.

\bibitem{BrKn86}
E.~Brieskorn and H.~Kn\"{o}rrer.
\newblock {\em Plane Algebraic Curves}.
\newblock Birkha\"{u}ser, 1986.

\bibitem{CaKa90}
D.~Cantor and E.~Kaltofen.
\newblock On fast multiplication of polynomials over arbitrary algebras.
\newblock {\em Acta Informatica}, 28(7):693--701, 1990.

\bibitem{Ca86}
J.~Cassel.
\newblock {\em Local Fields}, volume~3 of {\em Student Texts}.
\newblock LMS, 1986.

\bibitem{Ch51}
C.~Chevalley.
\newblock {\em Introduction to the Theory of Algebraic Functions of One
  Variable}, volume~6 of {\em Mathematical Surveys}.
\newblock AMS, 1951.

\bibitem{DaMaScXi05}
X.~Dahan, E.~Schost, M.~M. Maza, W.~Wu, and Y.~Xie.
\newblock On the complexity of the {D}5 principle.
\newblock {\em SIGSAM Bull.}, 39(3):97--98, 2005.

\bibitem{DeDiDu85}
J.~{Della Dora}, C.~Dicrescenzo, and D.~Duval.
\newblock About a new method for computing in algebraic number fields.
\newblock In {\em EUROCAL 85}. Springer-Verlag LNCS 204, 1985.

\bibitem{Du89}
D.~Duval.
\newblock Rational {P}uiseux expansions.
\newblock {\em Compositio Math.}, 70(2):119--154, 1989.

\bibitem{mort-Marc}
D.~Duval and A.~Poteaux.
\newblock Death of marc rybowicz, aged 52.
\newblock {\em ACM Commun. Comput. Algebra}, 50(4):191--191, Feb. 2017.

\bibitem{Ei66}
M.~Eichler.
\newblock {\em Introduction to the Theory of Algebraic Numbers and Functions}.
\newblock Academic Press, 1966.

\bibitem{GaGoPo17}
E.~R. Garc{\'i}a~Barroso, P.~D. Gonz{\'a}lez~P{\'e}rez, and P.~Popescu-Pampu.
\newblock Variations on inversion theorems for newton--puiseux series.
\newblock {\em Mathematische Annalen}, 368(3):1359--1397, Aug 2017.

\bibitem{GaGe13}
J.~v.~z. Gathen and J.~Gerhard.
\newblock {\em Modern Computer Algebra}.
\newblock Cambridge University Press, New York, NY, USA, 3rd edition, 2013.

\bibitem{HuPa98}
X.~Huang and V.~Y. Pan.
\newblock Fast rectangular matrix multiplication and applications.
\newblock {\em Journal of Complexity}, 14(2):257 -- 299, 1998.

\bibitem{KaSa99}
F.~Kako and T.~Sasaki.
\newblock Solving multivariate algebraic equations by {H}ensel construction.
\newblock {\em Japan J. of Industrial and Applied Math.}, 16:257--285, 1999.

\bibitem{Ka88}
E.~Kaltofen.
\newblock Greatest common divisors of polynomials given by straight-line
  programs.
\newblock {\em J. ACM}, 35(1):231--264, Jan. 1988.

\bibitem{KuTr78}
H.~T. Kung and J.~F. Traub.
\newblock All algebraic functions can be computed fast.
\newblock {\em Journal of the Association for Computing Machinery},
  25(2):245--260, 1978.

\bibitem{Le14}
F.~Le~Gall.
\newblock Powers of tensors and fast matrix multiplication.
\newblock In {\em Proceedings of the 39th International Symposium on Symbolic
  and Algebraic Computation}, ISSAC '14, pages 296--303, New York, NY, USA,
  2014. ACM.

\bibitem{MoSc16}
G.~Moroz and {\'E}.~Schost.
\newblock A fast algorithm for computing the truncated resultant.
\newblock In {\em ISSAC '16: Proceedings of the twenty-first international
  symposium on Symbolic and algebraic computation}, pages 1--8, New York, NY,
  USA, 2016. ACM.

\bibitem{Mu75}
D.~R. Musser.
\newblock Multivariate polynomial factorization.
\newblock {\em J. ACM}, 22(2):291--308, Apr. 1975.

\bibitem{CoSiTrUl02}
B.~M.~T. O.~Cormier, M. F.~Singer and F.~Ulmer.
\newblock Linear differential operators for polynomial equations.
\newblock {\em J. Symbolic Comput.}, 34(5):355--398, 2002.

\bibitem{PaSc06}
C.~Pascal and E.~Schost.
\newblock Change of order for bivariate triangular sets.
\newblock In {\em ISSAC'06}, pages 277--284. ACM, 2006.

\bibitem{Mo99}
J.~M. Peral.
\newblock {\em Pol\'igonos de newton de orden superior y aplicaciones
  aritm\'eticas}.
\newblock PhD thesis, Universitat de Barcelona, 1999.

\bibitem{Po08}
A.~Poteaux.
\newblock {\em Calcul de d\'eveloppements de Puiseux et application au calcul
  de groupe de monodromie d'une courbe alg\'ebrique plane}.
\newblock PhD thesis, Universit\'e de Limoges, 2008.

\bibitem{PoRy08}
A.~Poteaux and M.~Rybowicz.
\newblock Good reduction of puiseux series and complexity of the newton-puiseux
  algorithm over finite fields.
\newblock In {\em ISSAC '08: Proceedings of the twenty-first international
  symposium on Symbolic and algebraic computation}, pages 239--246, New York,
  NY, USA, 2008. ACM.

\bibitem{PoRy11}
A.~Poteaux and M.~Rybowicz.
\newblock Complexity bounds for the rational newton-puiseux algorithm over
  finite fields.
\newblock {\em Applicable Algebra in Engineering, Communication and Computing},
  22:187--217, 2011.
\newblock 10.1007/s00200-011-0144-6.

\bibitem{PoRy12}
A.~Poteaux and M.~Rybowicz.
\newblock Good reduction of puiseux series and applications.
\newblock {\em Journal of Symbolic Computation}, 47(1):32 -- 63, 2012.

\bibitem{PoRy15}
A.~Poteaux and M.~Rybowicz.
\newblock Improving complexity bounds for the computation of puiseux series
  over finite fields.
\newblock In {\em Proceedings of the 2015 ACM on International Symposium on
  Symbolic and Algebraic Computation}, ISSAC '15, pages 299--306, New York, NY,
  USA, 2015. ACM.

\bibitem{PoSc13b}
A.~Poteaux and E.~Schost.
\newblock On the complexity of computing with zero-dimensional triangular sets.
\newblock {\em Journal of Symbolic Computation}, 50(0):110 -- 138, 2013.

\bibitem{ScSt71}
A.~Sch\"onage and V.~Strassen.
\newblock Schnelle multiplikation gro\ss er zahlen.
\newblock {\em Computing 7}, pages 281--292, 1971.

\bibitem{Sh94}
V.~Shoup.
\newblock Fast construction of irreducible polynomials over finite fields.
\newblock {\em Journal of Symbolic Computation}, 17:371--391, 1994.

\bibitem{Te73}
B.~Teissier.
\newblock Cycles \'evanescents, sections planes et conditions de whitney.
\newblock In {\em Singularit\'es \`a Carg\`ese}, number 7-8 in Ast\'erisque,
  pages 285--362. Soci\'et\'e math\'ematique de France, 1973.

\bibitem{HoLe18}
J.~Van Der~Hoeven and G.~Lecerf.
\newblock {Accelerated tower arithmetic}.
\newblock Preprint, 2018.

\bibitem{Wa50}
R.~J. Walker.
\newblock {\em Algebraic Curves}.
\newblock Springer-Verlag, 1950.

\bibitem{We16}
M.~Weimann.
\newblock Bivariate factorization using a critical fiber.
\newblock {\em Journal of Foundations of Computational Mathematics}, pages
  1--45, 2016.

\end{thebibliography}
